\documentclass[leqno,12pt]{amsart} %leqno is the option to put formula numbers on the left side
\usepackage{amssymb}
\usepackage{amsmath}
\usepackage{amsthm}
\usepackage{amscd}
\usepackage{graphicx}
\usepackage[dvips]{color}
\usepackage[all]{xy}
\theoremstyle{plain} %text of this environment is typesetted in italics
\newtheorem{theorem}{\indent\sc Theorem}[section]
\newtheorem{lemma}[theorem]{\indent\sc Lemma}
\newtheorem{corollary}[theorem]{\indent\sc Corollary}
\newtheorem{proposition}[theorem]{\indent\sc Proposition}
\newtheorem{claim}[theorem]{\indent\sc Claim}

\theoremstyle{definition} %text of this environment is typesetted in roman letters
\newtheorem{definition}[theorem]{\indent\sc Definition}
\newtheorem{remark}[theorem]{\indent\sc Remark}
\newtheorem{example}[theorem]{\indent\sc Example}

\begin{document}

\title[Toward a higher codimensional Ueda theory]
{Toward a higher codimensional Ueda theory} %title of paper and the running head option

\author[T. Koike]{Takayuki Koike} %first author's name and the running head option

%%%%%%%%%%%%%%% footnote %%%%%%%%%%%%%%%%
%In case \subjclass[2010] command is not effective
%(or the version of amsart.cls is old), write as follows instead:
%\renewcommand{\thefootnote}{\fnsymbol{footnote}}
%\footnote[0]{2010\textit{ Mathematics Subject Classification}.
%Primary 00; Secondary 00.}
%
%\thanks{ %acknowledgment of support etc. if any
%$^{*}$Thanks.
%}
%%%%%%%%%%%% Authors' addresses %%%%%%%%%%%%%
\address{% First Author
Graduate School of Mathematical Sciences, The University of Tokyo \endgraf
3-8-1 Komaba, Meguro-ku, Tokyo, 153-8914 \endgraf
Japan
}
\email{tkoike@ms.u-tokyo.ac.jp}
%%%%%%%%%%%%%%%%%%%%%%%%%%%%%%%%%%%%%%%%%

\maketitle

\begin{abstract}
Ueda's theory is a theory on a flatness criterion around a smooth hypersurface of a certain type of topologically trivial holomorphic line bundles. 
We propose a codimension two analogue of Ueda's theory. 
As an application, we give a sufficient condition for the anti-canonical bundle of the blow-up of the three dimensional projective space at $8$ points to be non semi-ample however admit a smooth Hermitian metric with semi-positive curvature. 
\end{abstract}

\section{Introduction}
Ueda's theory is a theory on a flatness criterion around a smooth hypersurface of a certain type of topologically trivial holomorphic line bundles. 
We propose a codimension two analogue of Ueda's theory.  
Namely we shall describe a sufficient condition for the line bundle $\mathcal{O}_X(S)$ to be flat on a neighborhood of $C$ in $X$, where $S$ is a smooth hypersurface of a complex manifold $X$ and $C$ is a compact smooth hypersurface of $S$. 

Let us recall briefly Ueda's theory. 
Let $S$ be a smooth compact hypersurface of a complex manifold $X$ with the topologically trivial normal bundle $N_{S/X}$. 
Then $\mathcal{O}_X(S)$ is topologically trivial on a tubular neighborhood of $S$ in $X$. 
For such a pair $(S, X)$, Ueda defined an obstruction class $u_n(S, X)\in H^1(S, N_{S/X}^{-n})$ for each $n\geq 1$, which enjoys the property that $\mathcal{O}_X(S)$ is not flat around $S$ if there exists an integer $n\geq 1$ such that $u_n(S, X)\not=0$ (\cite{U}, \cite{N}, see also \S 2.1 here). 
By using these obstruction classes, he gave a sufficient condition for $\mathcal{O}_X(S)$ to be flat around $S$ as follows. 
Let $\mathcal{E}_0(S)$ be the set of the all torsion elements of ${\rm Pic}^0(S)$. 
Fix an invariant distance $d$ of ${\rm Pic}^0(S)$ (i.e. $d$ is a distance of ${\rm Pic}^0(S)$ such that, for each $E_1, E_2$, and $G\in {\rm Pic}^0(S)$, $d(E_1^{-1}, E_2^{-1})=d(E_1, E_2)$ and  $d(E_1\otimes G, E_2\otimes G)=d(E_1, E_2)$ hold). 
By using this distance $d$, define 
\[
\mathcal{E}_1(S):=\{E\in{\rm Pic}^0(S)\mid \exists \alpha\in\mathbb{R}_{>0}\ {\rm s.t.}\ \forall n\in\mathbb{Z}_{>0},\ d(\mathcal{O}_C, E^n)\geq (2n)^{-\alpha}\}. 
\]
Clearly this definition of the set $\mathcal{E}_1(S)$ does not depend on the choice of an invariant distance $d$. 
Note that the Lebesgue measure of ${\rm Pic}^0(S)\setminus \mathcal{E}_1(S)$ is zero, however $\mathcal{E}_1(S)$ is the union of countably many nowhere dense closed subsets of ${\rm Pic}^0(S)$. 
Ueda showed that $\mathcal{O}_X(S)$ is flat around $S$ 
when $N_{S/X}$ is an element of $\mathcal{E}_0(S)\cup\mathcal{E}_1(S)$ and $u_n(S, X)=0$ holds for each $n\geq 1$ (\cite[Theorem 3]{U}, see also Theorem \ref{thm:ueda_3} here). 

Now let us return to our subject. 
Let $X$ be a complex manifold, $S$ a smooth hypersurface of $X$, and $C$ be a smooth compact hypersurface of $S$. 
Assume that the normal bundle $N_{S/X}$ is flat around $C$. 
For such a triple $(C, S, X)$, we define a new obstruction class $u_{n, m}(C, S, X)\in H^1(C, N_{S/X}|_C^{-n}\otimes N_{C/S}^{-m})$ for each $n\geq 1, m\geq 0$ as a codimension two version of Ueda's obstruction classes (see \S 3 for the definition of $u_{n, m}(C, S, X)$). 
These new obstruction classes enjoy the property that $\mathcal{O}_X(S)$ is not flat around $C$ if there exists a pair of integers $n\geq 1$ and $m\geq 0$ such that $u_{n, m}(C, S, X)\not=0$.
By using these obstruction classes, we can describe a sufficient condition for $\mathcal{O}_X(S)$ to be flat around $C$ as follows. 

\begin{theorem}\label{main}
Let $X$ be a complex manifold, $S$ a smooth hypersurface of $X$, and $C$ be a smooth compact K\"ahler hypersurface of $S$ such that $N_{S/X}|_V$ is flat, where $V$ is a sufficiently small neighborhood of $C$ in $S$. 
Assume one of the following three conditions holds: 
$(i)$ $N_{C/S}\in\mathcal{E}_0(C)$ and $N_{S/X}|_C\in\mathcal{E}_0(C)$, 
$(ii)$ $N_{C/S}=N_{S/X}|_C\in\mathcal{E}_1(C)$, 
$(iii)$ $N_{S/X}|_C\in\mathcal{E}_0(C)$ and there exists a strongly $1$-convex neighborhood $V$ of $C$ in $S$ such that $C$ is the maximal compact analytic subset of $V$. 
Further assume that $u_{n, m}(C, S, X)=0$ holds for all  $n\geq 1, m\geq 0$. 
Then there exists a neighborhood $W$ of $C$ in $X$ such that $\mathcal{O}_X(S)|_W$ is flat. 
\end{theorem}

Note that, when $C$ is a curve, the condition $(iii)$ above is equivalent to the condition that 
$N_{C/S}$ is negative and $N_{S/X}|_C\in\mathcal{E}_0(C)$ (see \S 2.2 here for the details). 
We will prove Theorem \ref{main} by considering a codimension two analogue of the argument used in the proof of Ueda's theorem \cite[Theorem 3]{U}. 

Let us explain our motivation here. 
Let $X$ be a smooth projective manifold and $S$ be a hypersuface of $X$ such that the line bundle $\mathcal{O}_X(S)$ is nef. 
Our original interest is the existence (or non-existence) of smooth Hermitian metrics on $\mathcal{O}_X(S)$ with semi-positive curvature. 
When the base locus $C:=\mathbb{B}(S)$ of the linear system $|S|$ is a hypersurface of $X$, we gave some sufficient conditions for $\mathcal{O}_X(S)$ to (or not to) admit a smooth Hermitian metric with semi-positive curvature 
by considering some flatness criteria for $\mathcal{O}_X(S)$ around $C$ in \cite{K2}, \cite{K3} (Note that \cite{K3} is essentially based on Ueda's theory). 
Especially, in \cite[3.4]{K2}, we showed that $\mathcal{O}_X(S)$ admits a smooth Hermitian metric with semi-positive curvature when $\mathcal{O}_X(S)$ is flat around $C$. 
Now let us consider the case where the codimension of $C$ is greater than one. 
In this case, we can also apply the same argument as in \cite[3.4]{K2} when $\mathcal{O}_X(S)$ is flat around $C$ (see the proof of Corollary \ref{main_cor} in \S5.2 here). 
Thus we need a flatness criterion for $\mathcal{O}_X(S)$ around $C$, 
which is the motivation for considering the situation as in Theorem \ref{main}. 

One of the most important applications of Ueda's theory to algebro-geometric situations is on the semi-positivity of  the anti-canonical bundle of the blow-up of $\mathbb{P}^2$ at $9$ points (\cite{B} and \cite{U}, see also \S 5 here). 
As an analogue, the following corollary follows from Theorem \ref{main}. 

\begin{corollary}\label{main_cor}
Let $C_0\subset\mathbb{P}^3$ be a complete intersection of two quadric surfaces of $\mathbb{P}^3$ and let $p_1, p_2, \dots, p_8\in C_0$ be $8$ points different from each other. 
Assume $\mathcal{O}_{\mathbb{P}^3}(-2)|_{C_0}\otimes \mathcal{O}_{C_0}(p_1+p_2+\dots+p_8)\in \mathcal{E}_1(C_0)$. 
Then the anti-canonical bundle of the blow-up of $\mathbb{P}^3$ at $\{p_j\}_{j=1}^8$ is not semi-ample, however admits a smooth Hermitian metric with semi-positive curvature. 
\end{corollary}

Note that, when $\mathcal{O}_{\mathbb{P}^3}(-2)|_{C_0}\otimes \mathcal{O}_{C_0}(p_1+p_2+\dots+p_8)\in \mathcal{E}_0(C_0)$, 
the anti-canonical line bundle of the blow-up of $\mathbb{P}^3$ at $\{p_j\}_{j=1}^8$ is semi-ample, and thus it admits a smooth Hermitian metric with semi-positive curvature. 
When $\mathcal{O}_{\mathbb{P}^3}(-2)|_{C_0}\otimes \mathcal{O}_{C_0}(p_1+p_2+\dots+p_8)\not\in \mathcal{E}_0(C_0)$, 
the stable base locus of the anti-canonical line bundle is the strict transform of $C_0$ and thus it is not semi-ample (however it is nef). 
Note also that it is shown by Lesieutre and Ottem that the anti-canonical bundle $-K_X$ of the blow-up of $\mathbb{P}^3$ at very general $8$ points is a nef line bundle such that the set of curves $C$ with $-K_X. C=0$ is countably infinite \cite{LO}, 
and thus it gives an affirmative answer to the question of Totaro \cite{T}. 
For the details of this example, see \S 5 here. 

The organization of the paper is as follows. 
In \S 2, we will review Ueda's theory and give an explanation on some fundamental results which will be needed in the proof of Theorem \ref{main}. 
In \S 3, the obstruction class $u_{n, m}(C, S, X)$ will be defined for each $n\geq 1, m\geq 0$, and some fundamental properties of them will be shown. 
In \S 4, Theorem \ref{main} will be proven. 
In \S 5, some examples will be given and Corollary \ref{main_cor} will be proven. 

\vskip3mm
{\bf Acknowledgment. } 
The author would like to give heartful thanks to his supervisor Prof. Shigeharu Takayama whose comments and suggestions were of inestimable value for my study. 
He also thanks Prof. Tetsuo Ueda and Prof. Yoshinori Gongyo for helpful comments and warm encouragements. 
He is supported by the Grant-in-Aid for Scientific Research (KAKENHI No.25-2869) and the Grant-in-Aid for JSPS fellows. 
This work is supported by the Program for Leading Graduate
Schools, MEXT, Japan. 

\section{Preliminaries}
\subsection{A review of Ueda's theory}

Let $X$ be a complex manifold and $L$ be a holomorphic line bundle on $X$. 
We say that $L$ is a flat line bundle if each of the transition functions $t_{jk}\in \Gamma(U_{jk}, \mathcal{O}_{U_{jk}}^*)$ of $L$ can be taken as a complex constant with modulus $1$, 
where $\{U_j\}$ is a suitable open covering of $X$ and $U_{jk}:=U_j\cap U_k$. 
This condition is equivalent to the condition that $L$ can be regarded as an element of $H^1(X, U(1))\ (\subset H^1(X, \mathcal{O}_X^*))$, where $U(1):=\{z\in \mathbb{C}\mid |z|=1\}$. 
It can be said directly from the definition that a flat line bundle admits 
a metric which can locally be regarded as a constant function by using a suitable local frame of the line bundle (we call such a metric a flat metric). 
Since the (Chern) curvature tensor of a flat metric is $0$, 
it can be said that the first Chern class of a flat line bundle is trivial, 
and thus a flat line bundle is topologically trivial. 
When the manifold $X$ is compact and K\"ahler, it is known that the converse also holds (see \cite[\S 1]{U} for example). 
However, in general, a topologically trivial holomorphic line bundle need not admit the flat structure. 

Now let us start reviewing Ueda's theory along \cite[\S2]{U} for a complex manifold $X$ and a smooth hypersurface $S\subset X$ whose normal bundle $N_{S/X}$ is flat. 
Take a sufficiently fine open covering $\{V_j\}$ of $S$. 
From the assumption, 
$N_{C/X}=\{(V_{jk}, t_{jk})\}$
holds in $H^1(C, U(1))$ for some constants $t_{jk}\in U(1)$, where $V_{jk}:=V_j\cap V_k$. 
Let $W$ be a sufficiently small tubular neighborhood of $S$ in $X$ and $\{W_j\}$ be a sufficiently fine open covering of $W$. Without loss of generality, we may assume that the index sets of $\{V_j\}$ and $\{W_j\}$ coincide and $W_j\cap S=V_j$ holds. 
We choose local coordinates $(w_j, z_j)$ of $W_j$ satisfying the conditions that $z_j$ is a coordinate of $V_j$, $\{w_j=0\}=V_j$ holds on $W_j$, and that $(w_j/w_k)|_{V_{jk}}\equiv t_{jk}$ holds on $V_{jk}$ for all $j$ and $k$. 
Let $n$ be a positive integer. 
We say that a system $\{(W_j, w_j)\}$ is {\it of order $n$} if ${\rm mult}_{V_{jk}}(t_{jk}w_k-w_j)\geq n+1$ holds on each $W_{jk}$. When there exists a system $\{(W_j, w_j)\}$ of order $n$, the Taylor expansion of $t_{jk}w_k$ for the variable $w_j$ on $W_{jk}$ around $w_j=0$ can be written in the form 
\[
t_{jk}w_k=w_j+f_{jk}^{(n+1)}(z_k)\cdot w_j^{n+1}+O(w_j^{n+2})
\]
for some holomorphic function $f^{(n+1)}_{jk}$ defined on $V_{jk}$. 
Here we remark that, for all $m>n$, a system $\{(W_j, w_j)\}$ of order $m$ is also a system of order $n$ and in this case the function $f^{(m+1)}_{jk}$ is the constant function $0$. 
It is known that $\{(V_{jk}, f^{(n+1)}_{jk})\}$ satisfies the cocycle condition (\cite[p. 588]{U}, see also the proof of Proposition \ref{prop:cocycle} here). 

\begin{definition}
Suppose that there exists a system of order $n$. 
Then the cohomology class 
\[
u_n(S, X):=\{(V_{jk}, f^{(n+1)}_{jk})\}\in H^1(S, N_{S/X}^{-n})
\]
is called {\it the $n$-th Ueda class} of the pair $(S, X)$. 
\end{definition}

The $n$-th Ueda class does not depend on the choice of local coordinates system up to non-zero constant multiples (\cite[1.3]{N}, see also the proof of Proposition \ref{prop:u_nm_well-def} here). 
It is known that $u_n(S, X)=0$ if and only if there exists a system of order $n+1$. 
Thus only one phenomenon of the following occurs. 
\begin{itemize}
\item There exists an integer $n\in\mathbb{Z}_{>0}$ such that $u_m(S, X)$ can be defined only when $m\leq n$, $u_m(S, X)=0$ holds for all $m<n$, and $u_n(S, X)\not=0$ holds. 
\item For every integer $n\in\mathbb{Z}_{>0}$, $u_n(S, X)$ can be defined and it is equal to zero. 
\end{itemize}

\begin{definition}[{\cite[p. 589]{U}}]\label{type}
We denote ``${\rm type}\,(S, X)=n$'' and say that the pair $(S, X)$ is {\it of finite type} when $u_m(S, X)$ can be defined only when $m\leq n$, $u_m(S, X)=0$ holds for all $m<n$, and $u_n(S, X)\not=0$ holds.  
In the other case, we denote ``${\rm type}\,(S, X)=\infty$'' and say that the pair $(S, X)$ is {\it of infinite type}. 
\end{definition}

Ueda showed the following theorem. 

\begin{theorem}[{\cite[Theorem 3]{U}}]\label{thm:ueda_3}
Let $X$ be a complex manifold and $S$ be a smooth compact K\"ahler hypersurface of $X$. 
Assume that $N_{S/X}\in\mathcal{E}_0(S)\cup\mathcal{E}_1(S)$ and ${\rm type}\,(S, X)=\infty$. 
Then there exists a neighborhood $V$ of $S$ in $X$ such that the line bundle $\mathcal{O}_V(S)$ is flat. 
\end{theorem}

\begin{remark}
In \cite{U}, the above Theorem \ref{thm:ueda_3} is stated only for the case where $X$ is a surface. 
However, the proof of \cite[Theorem 3]{U} does not depends on the dimensions of $S$ and $X$, 
thus we obtain the above Theorem \ref{thm:ueda_3}. 
\end{remark}

\subsection{Some fundamental results}

In this subsection, we give some preliminary results needed in the proof of Theorem \ref{main}. 
We first give the definition of the exceptionality in the sense of Grauert. 

\begin{definition}[{\cite{G}, see also \cite[2.6]{CM}}]\label{def:excep}
Let $C$ be a compact connected subvariety of a complex manifold $S$. 
We say that $C$ is an {\it exceptional subvariety of $S$ in the sense of Grauert} if there exists a strongly strongly $1$-convex neighborhood $V$ of $C$ in $S$ such that $C$ is the maximal compact analytic subset of $V$. 
\end{definition}

When $S$ is a smooth surface and $C$ is a smooth curve embedded in $S$, 
it is known that $C$ is an exceptional subvariety of $S$ in the sense of Grauert 
if and only if the normal bundle $N_{C/S}$ is negative (\cite[4.9]{L}, see also the last of Chapter 2 of \cite{CM}). 
%It is known that $C$ is an exceptional subvariety of $S$ in the sense of Grauert 
%if the normal bundle $N_{C/S}$ is negative \cite[Satz 8]{G}. 
%It is also known that the converse holds 
%when $S$ is a smooth surface and $C$ is a smooth curve embedded in $S$ 
%it is also known that $C$ is an exceptional subvariety of $S$ in the sense of Grauert 
%if and only if the normal bundle $N_{C/S}$ is negative 
%(\cite[4.9]{L}, see also \cite[1.5.3]{CM}). 
In order to deal with the situation of $(iii)$ in Theorem \ref{main}, we use the following form of Rossi's theorem. 

\begin{proposition}[{a version of Rossi's theorem \cite[3]{R}, see also \cite[3.1 (1)]{K2}}]\label{prop:Rossi}
Let $C$ be a compact connected subvariety of a complex manifold $S$, 
and $V$ be a strongly pseudoconvex neighborhood of $C$ in $S$ such that $C$ is 
the maximal compact analytic subset of $V$. 
Then for each coherent sheaf $\mathcal{S}$ on $V$, there exists an integer $N(\mathcal{S})$ such that the natural map $H^1(V, \mathcal{S})\to H^1(V, \mathcal{S}\otimes\mathcal{O}_{V}/I_C^n)$ is injective for each integer $n\geq N(\mathcal{S})$, where $I_C$ is the defining ideal sheaf of $C\subset V$. 
\end{proposition}

In the proof of Theorem \ref{main}, we will use the following Lemma \ref{lem:KS_2}, which is a variant of \cite[Lemma 2]{KS} (see also \cite[Lemma 3]{U}). 

\begin{lemma}\label{lem:KS_2}
Let  $C$ be a compact complex manifold embedded in a complex manifold $S$. 
Fix a sufficiently small connected neighborhood $V$ of $C$ in $S$ and a sufficiently fine open covering $\{V_j\}$ of $V$ which consists of a finite number of open sets. 
Fix also be a relatively compact open domain $V_0\subset V$ which contains $C$. 
For each flat line bundle $E$ on $V$, there exists a positive constant $K=K(E)$ depending only on $E$ such that, 
for each $1$-cocycle $\alpha=\{(V_{jk}, \alpha_{jk})\}$ of $E$ which can be realized as the coboundary of some $0$-cochain, there exists a $0$-cochain $\beta=\{(V_j, \beta_j)\}$ of $E$ such that $\alpha$ is equal to the coboundary $\delta (\beta)$ of $\beta$ and the inequality 
\[
\max_j\sup_{V_0\cap V_j}|\beta_j|\leq K\cdot\max_{jk}\sup_{V_0\cap V_{jk}}|\alpha_{jk}|
\]
holds. 
\end{lemma}

Whereas \cite[Lemma 2]{KS} (\cite[Lemma 3]{U}) is on the existence of such a constant $K=K(E)$ as in Lemma \ref{lem:KS_2} for a flat line bundle $E$ on a compact complex manifold,  
Lemma \ref{lem:KS_2} is on the existence of $K=K(E)$ for $E$ defined on a neighborhood of a compact complex manifold, which is not compact. 
However, since $V$ can be covered by finitely many sufficiently fine open subsets, 
Lemma \ref{lem:KS_2} can be showed by the same argument as in \cite[\S 6]{KS}. 
As the proof of Theorem \ref{thm:ueda_3} in \cite{U}, 
the proof of Theorem \ref{main} is also based on the following lemmata. 

%\begin{lemma}[{Kodaira-Spencer \cite{KS}, see also \cite[Lemma 3]{U}}]\label{lem:KS}
%Let  $C$ be a compact complex manifold and $E$ be a flat line bundle on $C$. 
%Fix a sufficiently fine open covering $\{U_j\}$ of $C$ which consists of a finite number of open sets. 
%Then there exists a positive constant $K=K(E)$ depending only on $E$ such that, 
%for each $1$-cocycle $\alpha=\{(U_{jk}, \alpha_{jk})\}$ of $E$ which can be realized as the coboundary of some $0$-cochain, there exists a $0$-cochain $\beta=\{(U_j, \beta_j)\}$ of $E$ such that $\alpha$ is equal to the coboundary $\delta (\beta)$ of $\beta$ and the inequality 
%\[
%\max_j\sup_{U_j}|\beta_j|\leq K\cdot\max_{jk}\sup_{U_{jk}}|\alpha_{jk}|
%\]
%holds. 
%\end{lemma}
%From now on, we will use the invariant distance $d$ of ${\rm Pic}^0(C)$ defined by 
%$d(\mathcal{O}_C, E):=\inf\max_{j, k} |1-\sigma_{jk}|$, where the infimum is taken over all systems $\{(U_{jk}, \sigma_{jk})\}$ which can be regarded as a system of transition functions of $E\in{\rm Pic}^0(C)$ (see also \cite[p. 601]{U} for the details of this distance $d$). 
\begin{lemma}[{\cite[Lemma 4]{U}}]\label{lem:ueda_4}
Let  $C$ be a compact complex manifold. 
Fix a sufficiently fine open covering $\{U_j\}$ of $C$ which consists of a finite number of open sets. 
Then there exists a positive constant $K$ such that, 
for each flat line bundle $E$ on $C$, the following condition holds: 
For each $1$-cocycle $\alpha=\{(U_{jk}, \alpha_{jk})\}$ such that the element of $H^1(C, E)$ defined by $\alpha$ is the trivial one, there exists a $0$-cochain $\beta=\{(U_j, \beta_j)\}$ of $E$ such that $\delta(\beta)=\alpha$ and 
\[
d(\mathcal{O}_C, E)\cdot\max_j\sup_{U_j}|\beta_j|\leq K\cdot\max_{jk}\sup_{U_{jk}}|\alpha_{jk}|
\]
hold. 
\end{lemma}

\begin{lemma}[{\cite{S}, see also \cite[Lemma 5]{U}}]\label{lem:siegel}
Let $\{\varepsilon_\lambda\}_{\lambda\geq 1}$ be a series of positive numbers satisfying the following two conditions:  the condition that there exists a positive number $\alpha$ such that, for each $\lambda$, $\varepsilon_\lambda<(2\lambda)^\alpha$ holds, 
and the condition that $\varepsilon_{\nu-\mu}^{-1}\leq \varepsilon_\nu^{-1}+\varepsilon_\mu^{-1}$ holds for each $\nu>\mu$. 
Then for each power series 
\[
f(Z)=\sum_{\lambda=2}^\infty a_\lambda Z^\lambda
\]
with a positive radius of convergence, the formal power series $Z+\sum_{\lambda=2}^\infty c_\lambda Z^\lambda$ uniquely determined by 
\[
\sum_{\lambda=2}^\infty \varepsilon_{\lambda-1}^{-1}c_\lambda Z^\lambda
=f\left(Z+\sum_{\lambda=2}^\infty c_\lambda Z^\lambda\right)
\]
has a positive radius of convergence. 
\end{lemma}

\section{The definition and some basic properties of the class $u_{n, m}(C, S, X)$}

\subsection{The definition of the class $u_{n, m}(C, S, X)$ and ${\rm type}\,(C, S, X)$}

Let $X$ be a complex manifold, $S$ a smooth hypersurface of $X$, and $C$ be a smooth compact K\"ahler hypersurface of $S$ such that $N_{S/X}|_V$ is flat, where $V$ is a sufficiently small neighborhood of $C$ in $S$. 
Fix a sufficiently small tubular neighborhood $W$ of $C$ such that $W\cap S=V$. 
Fix also a sufficiently fine open covering $\{U_j\}$ of $C$, $\{V_j\}$ of $V$, and $W_j$ of $W$ such that $W_j\cap S=V_j$ and $V_j\cap C=U_j$ hold. 
Let $x_j$ be a coordinates system of $U_j$, $z_j$ a defining holomorphic function of $U_j$ in $V_j$, and $w_j$ be a defining holomorphic function of $V_j$ in $W_j$. 
Extending $x_j$ and $z_j$ to $W_j$ in a holomorphic way, we use a system $(x_j, z_j, w_j)$ as a coordinates system of $W_j$. 
Let $t_{jk}$ be a transition function of $N_{S/X}|_V$ on $V_{jk}$: i.e. $N_{S/X}|_V=\{(V_{jk}, t_{jk})\}$. 
As $N_{S/X}|_V$ is flat, $t_{jk}$ can be selected as a constant function on $V_{jk}$ with modulus $1$. 
From the same argument as in \cite[\S 2]{U}, it can be said that, 
without loss of generality, we may assume $({w_j}/{w_k})|_{V_{jk}}\equiv t_{jk}$ holds. 
Just from the same reason, we may also assume that $({z_j}/{z_k})|_{U_{jk}}\equiv s_{jk}$ holds, where $s_{jk}$ is a transition function of $N_{C/S}$ on $U_{jk}$. 

\begin{definition}
Let $n\geq 1$ and $m\geq 0$ be integers. 
A system $\{(W_j, w_j)\}$ is called {\it of order $(n, m)$} if, for each $j$ and $k$, the function $t_{jk}w_k-w_j$ on $W_{jk}$ satisfies that ${\rm mult}_{V_{jk}}(t_{jk}w_k-w_j)\geq n+1$ and that ${\rm mult}_{U_{jk}}((t_{jk}w_k-w_j)/w_j^{n+1})|_{V_{jk}}\geq m$. 
This condition is equivalent to the following condition: the coefficient of $z_j^\mu w_j^\nu$ in the Taylor expansion of the function $t_{jk}w_k-w_j$ in the variables $z_j$ and $w_j$ around $U_{jk}$ is equal to zero if $(\nu, \mu)\leq (n, m)$ holds in the lexicographical order, namely 
$(a, b)\leq (a', b')\overset{\mathrm{def}}{\Leftrightarrow} a< a'\ \text{or}\ ``a=a'\ \text{and}\ b\leq b'"$. 
\end{definition}

Assume that our system $\{(W_j, w_j)\}$ is of order $(n, m)$. 
Then the function $t_{jk}w_k$ can be expanded in the variable $w_j$ as follows: 
\begin{equation}\label{eq:f_def}
 t_{jk}w_k=w_j+f_{jk}^{(n+1)}(x_j, z_j)\cdot w_j^{n+1}+f_{jk}^{(n+2)}(x_j, z_j)\cdot w_j^{n+2}+\cdots . 
\end{equation}
Let 
\[
 f_{jk}^{(n+1)}(x_j, z_j)=g_{jk}^{(n+1, m)}(x_j)\cdot z_j^m+g_{jk}^{(n+1, m+1)}(x_j)\cdot z_j^{m+1}+\cdots 
\]
be the expansion of $f_{jk}^{(n+1)}(x_j, z_j)$ in the variable $z_j$. 

\begin{proposition}\label{prop:cocycle}
In the above setting, a system $\{(U_{jk}, g_{jk}^{(n+1, m)})\}$ enjoys the cocycle condition for the line bundle $N_{S/X}|_C^{-n}\otimes N_{C/X}^{-m}$. 
\end{proposition}

\begin{proof}
From the equation $(\ref{eq:f_def})$, 
\begin{eqnarray}
t_{jk}^{-n}w_k^{-n}&=& (w_j+f_{jk}^{(n+1)}(x_j, z_j)\cdot w_j^{n+1}+f_{jk}^{(n+2)}(x_j, z_j)\cdot w_j^{n+2}+\cdots)^{-n} \nonumber \\
&=& w_j^{-n}(1+f_{jk}^{(n+1)}(x_j, z_j)\cdot w_j^{n}+f_{jk}^{(n+2)}(x_j, z_j)\cdot w_j^{n+1}+\cdots)^{-n} \nonumber \\
&=& w_j^{-n}(1-nf_{jk}^{(n+1)}(x_j, z_j)\cdot w_j^{n}+O(w_j^{n+1})) \nonumber \\
&=& w_j^{-n}-nf_{jk}^{(n+1)}(x_j, z_j)+O(w_j) \nonumber
\end{eqnarray}
holds. Thus we obtain 
\begin{equation}\label{eq:mero_exp}
 \frac{1}{n}(w_j^{-n}-t_{jk}^{-n}w_k^{-n})|_{V_{jk}}=f_{jk}^{(n+1)}(x_j, z_j)=g_{jk}^{(n+1, m)}(x_j)\cdot z_j^m+g_{jk}^{(n+1, m+1)}(x_j)\cdot z_j^{m+1}+\cdots . 
\end{equation}
Therefore we can conclude that $\{(V_{jk}, f^{(n+1)}_{jk})\}$ satisfies the cocycle condition for the line bundle $N_{S/X}|_V^{-n}$: 
\[
 f_{jk}^{(n+1)}+t_{jk}^{-n}f_{kl}^{(n+1)}+t_{jl}^{-n}f_{lj}^{(n+1)}\equiv 0. 
\]
Regarding it as an equation of the local sections of $\mathcal{O}_{V}/\mathcal{O}_{V}(-(m+1)C)$, it follows that 
\[
 g_{jk}^{(n+1, m)}+t_{jk}^{-n}s_{jk}^{-m}g_{kl}^{(n+1, m)}+t_{jl}^{-n}s_{jl}^{-m}g_{lj}^{(n+1, m)}\equiv 0, 
\]
which shows the proposition. 
\end{proof}

\begin{definition}
Let $\{(W_j, w_j)\}$ be a system of order $(n, m)$. 
we denote by $u_{n, m}(C, S, X)$ the element of $H^1(C, N_{S/X}|_C^{-n}\otimes N_{C/S}^{-m})$ defined by the $1$-cocycle $\{(U_{jk}, g_{jk}^{(n+1, m)})\}$ 
and call it {\it the $(n, m)$-th Ueda class} of the triple $(C, S, X)$. 
\end{definition}

\begin{proposition}\label{prop:u_nm_well-def}
The above definition of the $(n, m)$-th Ueda class $u_{n, m}(C, S, X)$ of the triple $(C, S, X)$ is independent of the choice of a system of order $(n, m)$ up to non-zero constant multiples. 
Especially it can be said that the condition ``$u_{n, m}(C, S, X)=0$'' makes sense whenever there exits a system of order $(n, m)$. 
\end{proposition}

\begin{proof}
Take another system $\{(W_j, \widetilde{w}_j)\}$ of order $(n, m)$. 
Since both $\widetilde{w}_j$ and $w_j$ are the defining function of $V_j$ in $W_j$, 
there exists a nowhere vanishing holomorphic function $e_j$ defined on $W_j$ such that $\widetilde{w}_j=e_jw_j$. 
As $(\widetilde{w}_j/\widetilde{w}_k)|_{U_{jk}}=(w_j/w_k)|_{U_{jk}}(\equiv t_{jk})$, it holds that $e_j|_{U_{jk}}=e_k|_{U_{jk}}$ and thus $\{U_j, e_j|_{U_j}\}$ glues up to define a holomorphic function defined on the whole $C$. 
Therefore we can conclude that there exists a non-zero constant $M$ such that $e_j|_{U_{j}}\equiv M$ holds for all $j$. 
From this fact and the equation $(\ref{eq:mero_exp})$, it follows that the new $(n, m)$-th Ueda class defined by using the system $\{(W_j, \widetilde{w}_j)\}$ is just $M^{-n}$ times the original one defined by using $\{(W_j, w_j)\}$. 
\end{proof}

\begin{definition}
The $(n, m)$-th Ueda class $u_{n, m}(C, S, X)$ of the triple $(C, S, X)$ is said to be {\it well-defined} if there exits a system of order $(n, m)$. 
\end{definition}

From the definition, it is clear that, if $u_{n, m}(C, S, X)$ is well-defined, then $u_{\nu, \mu}(C, S, X)$ is also well-defined and is equal to zero for each $(\nu, \mu)$ less than $(n, m)$ in the lexicographical order. 
The following proposition is on the converse of it. 

\begin{proposition}\label{prop:newsystem}
Let $n\geq 1$ and $m\geq 0$ be integers. 
Assume one of the following three conditions holds: 
$(i')$ $N_{C/S}\in\mathcal{E}_0(C)$, 
$(ii)$ $N_{C/S}=N_{S/X}|_C\in\mathcal{E}_1(C)$, 
$(iii')$ $C$ is an exceptional subvariety of $S$ in the sense of Grauert. 
Further assume that 
$u_{\nu, \mu}(C, S, X)$ is well-defined and is equal to zero for each $(\nu, \mu)$ less than $(n, m)$ in the lexicographical order. 
Then $u_{n, m}(C, S, X)$ is also well-defined. 
\end{proposition}

The strategy of the proof of Proposition \ref{prop:newsystem} is almost the same as that of Theorem \ref{thm:ueda_3} in \cite{U}. We will prove it in the next subsection. 
By using Proposition \ref{prop:newsystem}, we can define ${\rm type}\,(C, S, X)$ as follows. 

\begin{definition}
Assume one of the following three conditions holds: 
$(i')$ $N_{C/S}\in\mathcal{E}_0(C)$, 
$(ii)$ $N_{C/S}=N_{S/X}|_C\in\mathcal{E}_1(C)$, 
$(iii')$ $C$ is an exceptional subvariety of $S$ in the sense of Grauert. 
We denote by ``${\rm type}\,(C, S, X)$'' the maximum in the lexicographical order of the set of all pairs $(n, m)$ such that the $(n, m)$-th Ueda class $u_{n, m}(C, S, X)$ of the triple $(C, S, X)$ is well-defined: i.e.\,${\rm type}\,(C, S, X)$ is the pair $(n, m)$ such that $u_{n, m}(C, S, X)$ is well-defined and non-trivial. 
We write ``${\rm type}\,(C, S, X)=\infty$'' if there is no such pair $(n, m)$ (i.e. if $u_{n, m}(C, S, X)$ is well-defined and $u_{n, m}(C, S, X)=0$ holds for each $n\geq 1$ and $m\geq 0$). 
\end{definition}

\subsection{Proof of Proposition \ref{prop:newsystem}}

Proposition \ref{prop:newsystem} directly follows from the following Lemmata \ref{lem:newsystem_1}, \ref{lem:newsystem_2}. 

\begin{lemma}\label{lem:newsystem_1}
Let $n\geq 1$ and $m\geq 0$ be integers. 
Assume that $u_{n, m}(C, S, X)$ is well-defined and is equal to zero. 
Then $u_{n, m+1}(C, S, X)$ is also well-defined. 
\end{lemma}

\begin{lemma}\label{lem:newsystem_2}
Let $n\geq 1$ be an integer. 
Assume one of the following three conditions holds: 
$(i')$ $N_{C/S}\in\mathcal{E}_0(C)$, 
$(ii)$ $N_{C/S}=N_{S/X}|_C\in\mathcal{E}_1(C)$, 
$(iii')$ $C$ is an exceptional subvariety of $S$ in the sense of Grauert. 
Further assume that 
$u_{n, m}(C, S, X)$ is well-defined and is equal to zero for each $m\geq 0$. 
Then $u_{n+1, 0}(C, S, X)$ is also well-defined. 
\end{lemma}

We will prove these lemmata in the following. 

\subsubsection{Proof of Lemma \ref{lem:newsystem_1}}

Fix a system $\{(W_j, w_j)\}$ of order $(n, m)$. 
From the assumption and the equation $(\ref{eq:mero_exp})$, there exists a system $\{(U_{j}, F_j(x_j))\}$ of holomorphic functions such that 
 \[
    (w_j^{-n}-t_{jk}^{-n}w_k^{-n})|_{V_{jk}}=F_jz_j^m-t_{jk}^{-n}s_{jk}^{-m}F_kz_k^m+O(z_j^{m+1})
 \]
holds on each $V_{jk}$. 
Define 
 \[
  \widetilde{w}_j:=w_j\cdot (1-F_j\cdot z_j^m\cdot w_j^n)^{-\frac{1}{n}}
 \]
by shrinking $V$ and $W$ if necessary. 
Since
 \[
  \widetilde{w}_j=w_j\cdot \left(1+\frac{1}{n}F_j\cdot z_j^m\cdot w_j^n+\cdots\right)
  =w_j+O(w_j^{n+2})
 \]
holds, it is clear that our new system $\{\widetilde{w}_j\}$ is also of order $(n, m)$. 
As $\widetilde{w}_j^{-n}=w_j^{-n}-F_jz_j^{m}$ 
holds, we obtain that
 \[
  \frac{1}{n}(\widetilde{w}_j^{-n}-t_{jk}^{-n}\widetilde{w}_k^{-n})|_{V_{jk}}=O(z_j^{m+1}), 
 \]
which means that the system $\{\widetilde{w}_j\}$ is of order $(n, m+1)$. 
\hfill $\Box$

\subsubsection{Proof of Lemma \ref{lem:newsystem_2} $(i')$}

We will show Lemma \ref{lem:newsystem_2} when the condition $(i')$ $N_{C/S}\in\mathcal{E}_0(C)$ holds. 
Fix a system $\{(W_j, w_j)\}$ of order $(n, 0)$. 
We will construct a new system $\{(W_j, \widetilde{w}_j)\}$ of order $(n+1, 0)$ by solving a functional equation 
\begin{equation}\label{eq:eqfunc1}
w_j=u_j+F_j(x_j, z_j)\cdot u_j^{n+1}
\end{equation}
with an uknown function $u_j$ on each $W_j$ after choosing a system of suitable holomorphic functions $\{(V_{jk}, F_j(x_j, z_j))\}$. 
We will define the function $F_j(x_j, z_j)$ by inductively defining the coefficient $\{G_j^{(m)}(x_j)\}$ 
of the variable $z_j^m$ in the expansion of $F_j(x_j, z_j)$: 
\[
F_j(x_j, z_j)=\sum_{m=0}^\infty G_j^{(m)}(x_j)\cdot z_j^m. 
\]
Fix  a positive constant $K(N_{S/X}|_C^{-n}\otimes N_{C/S}^{-m})$ as in \cite[Lemma 3]{U} for each $m\geq 0$ and 
set $K:=\max_mK(N_{S/X}|_C^{-n}\otimes N_{C/S}^{-m})$ (here the condition $(i')$ is needed for the existence of the maximum). 
Let 
\begin{equation}\label{eq:lem_tenkai}
t_{jk}w_k-w_j=f_{jk}^{(n+1)}(x_j, z_j)\cdot w_j^{n+1}+O(w_j^{n+2}),\ f_{jk}^{(n+1)}(x_j, z_j)=\sum_{m=0}^\infty g_{jk}^{(n+1, m)}(x_j)\cdot z_j^m
\end{equation}
be the expansion of $t_{jk}w_k-w_j$ on $W_{jk}$. 
Let $M$ and $R$ be sufficiently large positive number such that $\sup_{V_{jk}}|f^{(n+1)}_{jk}(x_j, z_j)|<M$ and $\{|z_j|<2R^{-1}, |w_j|<2R^{-1}\}\subset W_{j}$ hold for each $j, k$. 
Note that, from Cauchy's inequality, it holds that $|g^{(n+1, m)}_{jk}|<MR^{m}$. 
Consider the function 
\[
A(X):=\frac{KM}{1-(K+1)RX}=KM+KM(K+1)RX+KM(K+1)^2R^2X^2+\cdots
\]
and denote by $A_m$ the coefficient of $X^m$ in the Taylor expansion of $A(X)$ at $X=0$. 

\begin{lemma}\label{lem:newsystem_1_torsion}
There exists a system of functions $\{G_j^{(\mu)}(x_j)\}_{\mu=0}^\infty$ for each $j$ satisfying the following conditions. 
Let 
\begin{eqnarray}
G_j^{(\mu)}(x_j)\cdot z_j^\mu&=&G_j^{(\mu)}(x_j(x_k, z_k, w_k))\cdot z_j^\mu
=G_j^{(\mu)}(x_j(x_k, z_k, 0))\cdot z_j^\mu+O(w_k)\nonumber \\
&=&G_j^{(\mu)}(x_j(x_k, 0, 0))\cdot s_{jk}^{\mu}z_k^\mu+\sum_{\nu=1}^\infty G_{jk, \nu}^{(\mu)}(x_k)\cdot z_k^{\nu+\mu}+O(w_k) \nonumber
\end{eqnarray} 
be the expansion of $(G_j^{(\mu)}(x_j)\cdot z_j^\mu)|_{W_{jk}}$ in the variable $z_k$ and $w_k$ by regarding $G_j^{(\mu)}$ as a function defined on $W_j$ which does not depend on $z_j$ and $w_j$. 
Denote by $h_{1jk, \mu}(x_j)$ the coefficient of $z_j^\mu$ in the Taylor expansion of $\sum_{\mu=0}^{\infty}\sum_{\nu=1}^\infty G_{kj, \nu}^{(\mu)}(x_j)\cdot z_j^{\nu+\mu}$ at $z_j=0$ for each $\mu$. Then the coboundary $\delta\{(U_j, G_j^{(\mu)})\}$ is equal to $\{(U_{jk}, g_{jk}^{(n+1, \mu)}-t_{jk}^{-n}h_{1jk, \mu})\}$, and 
$\max_j\sup_{U_j}\left|G_j^{(\mu)}\right|\leq A_\mu$ for each $\mu\geq 0$. 
\end{lemma}

\begin{proof}
First we construct $\{G_j^{(0)}\}$ for each $j$. 
It follows from the definition that the system $\{(U_{jk}, g_{jk}^{(n+1, 0)})\}$ defines the cohomology class $u_{n, 0}(C, S, X)$, 
which is trivial  from the assumption of Lemma \ref{lem:newsystem_2}. 
Thus there exists an $1$-cochain $\{(U_j, G_j^{(0)})\}$ of $N_{S/X}|_C^{-n}$ such that
\[
\delta\{(U_j, G_j^{(0)})\}=\{(U_{jk}, g_{jk}^{(n+1, 0)})\},\ 
\max_j\sup_{U_j}\left|G_j^{(0)}\right|\leq K\cdot\max_{jk}\sup_{U_{jk}}\left|g_{jk}^{(n+1, 0)}\right|
\]
hold. 
Since $h_{1jk, 0}(x_j)\equiv 0$ and $K\cdot|g^{(n+1, 0)}_{jk}|\leq KM=A_0$ hold, $\{G_j^{(0)}\}$ is 
what we wanted. 

Next we will construct $\{G_j^{(m)}\}$ for each $j$ assuming that there exists $\{G_j^{(\mu)}\}$ for each $j$ and $\mu\leq m-1$ such that 
the coboundary $\delta\{(U_j, G_j^{(\mu)})\}$ is equal to $\{(U_{jk}, g_{jk}^{(n+1, \mu)}-t_{jk}^{-n}h_{1jk, \mu})\}$ and 
$\max_j\sup_{U_j}\left|G_j^{(\mu)}\right|\leq A_\mu$ hold. 
Let $\widetilde{w}_j$ be the solution of the functional equation 
\begin{equation}\label{eq:eqfunc2}
w_j=u_j+\left(\sum_{\mu=0}^{m-1} G_j^{(\mu)}(x_j)\cdot z_j^\mu\right)\cdot u_j^{n+1}
\end{equation}
with an unknown function $u_j$ (the existence of the solution is follows from the implicit function theorem). 
From the functional equation $(\ref{eq:eqfunc2})$, the function $(t_{jk}w_k)|_{W_{jk}}$ can be expanded as follows: 
\begin{eqnarray}
& &t_{jk}w_k \nonumber \\
&=& \left(\sum_{\mu=0}^{m-1} G_k^{(\mu)}(x_k)\cdot z_k^\mu\right)\cdot t_{jk}\widetilde{w}_k^{n+1}+t_{jk}\widetilde{w}_k \nonumber \\
&=& \left(\sum_{\mu=0}^{m-1} \left(G_k^{(\mu)}(x_k(x_j, 0, 0))\cdot s_{jk}^{-\mu}z_j^\mu+\sum_{\nu=1}^\infty G_{kj, \nu}^{(\mu)}(x_j)\cdot z_j^{\nu+\mu}\right)\right)\cdot t_{jk}\widetilde{w}_k^{n+1}\nonumber \\
&&+t_{jk}\widetilde{w}_k+O(\widetilde{w}_k^{n+2}) \nonumber 
\end{eqnarray}
\begin{eqnarray}
&=& \left(\sum_{\mu=0}^{m-1} \left(G_k^{(\mu)}(x_k(x_j, 0, 0))\cdot s_{jk}^{-\mu}+h_{1jk, \mu}\right)\cdot z_j^{\mu}+h_{1jk, m}\cdot z_j^m+O(z_j^{m+1})\right)\cdot t_{jk}\widetilde{w}_k^{n+1}\nonumber \\
&&+t_{jk}\widetilde{w}_k+O(\widetilde{w}_k^{n+2}) \nonumber \\
&=& \left(\sum_{\mu=0}^{m-1} \left(G_k^{(\mu)}(x_k(x_j, 0, 0))\cdot s_{jk}^{-\mu}+h_{1jk, \mu}\right)\cdot z_j^{\mu}+h_{1jk, m}\cdot z_j^m+O(z_j^{m+1})\right)\cdot t_{jk}^{-n}\widetilde{w}_j^{n+1}\nonumber \\
&&+t_{jk}\widetilde{w}_k+O(\widetilde{w}_j^{n+2}). \nonumber
\end{eqnarray}
Here we used the fact that $h_{1jk, \mu}$ depends only on $\{G_{j}^{(\mu')}\}_{\mu'=0}^{\mu-1}$. 
On the other hand, by using the equation $(\ref{eq:lem_tenkai})$, the function $(t_{jk}w_k)|_{W_{jk}}$  can also be expanded as follows: 
\begin{eqnarray}
& &t_{jk}w_k \nonumber \\
&=& w_j+\left(\sum_{\mu=0}^\infty g_{jk}^{(n+1, \mu)}(x_j)\cdot z_j^\mu\right)\cdot w_j^{n+1}+O(w_j^{n+2}) \nonumber \\
&=& \left(\sum_{\mu=0}^{m-1} \left(G_j^{(\mu)}(x_j)+g_{jk}^{(n+1, \mu)}(x_j)\right)\cdot z_j^\mu+g_{jk}^{(n+1, m)}(x_j)\cdot z_j^m +O(z_j^{m+1})\right)\cdot \widetilde{w}_j^{n+1}\nonumber \\
&&+\widetilde{w}_j+O(\widetilde{w}_j^{n+2}). \nonumber 
\end{eqnarray}
Thus we obtain 
\begin{eqnarray}
&& t_{jk}\widetilde{w}_k-\widetilde{w}_j \nonumber \\
&=&
-\left(\sum_{\mu=0}^{m-1} \left(G_k^{(\mu)}(x_k(x_j, 0, 0))\cdot s_{jk}^{-\mu}+h_{1jk, \mu}\right)\cdot z_j^{\mu}+h_{1jk, m}\cdot z_j^m+O(z_j^{m+1})\right)\cdot t_{jk}^{-n}\widetilde{w}_j^{n+1} \nonumber 
\end{eqnarray}
\begin{eqnarray}
&&+\left(\sum_{\mu=0}^{m-1} \left(G_j^{(\mu)}(x_j)+g_{jk}^{(n+1, \mu)}(x_j)\right)\cdot z_j^\mu+g_{jk}^{(n+1, m)}(x_j)\cdot z_j^m +O(z_j^{m+1})\right)\cdot \widetilde{w}_j^{n+1}\nonumber \\
&&+O(\widetilde{w}_j^{n+2}).  \nonumber 
\end{eqnarray}
Therefore, from the assumption of the induction, we can conclude that
\begin{eqnarray}
&& t_{jk}\widetilde{w}_k-\widetilde{w}_j \nonumber \\
&=&\left(\left(g_{jk}^{(n+1, m)}-t_{jk}^{-n}h_{1jk, m}\right)\cdot z_j^m +O(z_j^{m+1})\right)\cdot \widetilde{w}_j^{n+1}
+O(\widetilde{w}_j^{n+2}) \nonumber 
\end{eqnarray}
and thus it follows that the system $\{(U_{jk}, g_{jk}^{(n+1, m)}-t_{jk}^{-n}h_{1jk, m})\}$ is an $1$-cocycle which defines the $(n, m)$-th Ueda class, which is the trivial one from the assumption of Lemma \ref{lem:newsystem_2}. 
Hence it follows from \cite[Lemma 3]{U} that there exists a $0$-cochain $\{(U_j, G^{(m)}_j)\}$ such that $\delta\{(U_j, G^{(m)}_j)\}=\{(U_{jk}, g_{jk}^{(n+1, m)}-t_{jk}^{-n}h_{1jk, m})\}$ and 
\[
\max_j\sup_{U_j}\left|G_j^{(m)}\right|\leq K\cdot \max_{jk}\sup_{U_{jk}}\left|g_{jk}^{(n+1, m)}-t_{jk}^{-n}h_{1jk, m}\right| 
\]
hold. 
From the definition of $h_{1jk, m}$, we obtain the inequality 
\[
|h_{1jk, m}|\leq \text{the coefficient of}\ z_j^m\ \text{in the expansion of}\ 
\sum_{\mu=0}^{m-1}\sum_{\nu=1}^\infty  \left|G_{kj, \nu}^{(\mu)}(x_j)\right|\cdot z_j^{\nu+\mu}. 
\]
As it follows from the assumption of the induction and Cauchy's inequality (see Remark \ref{rmk:1}) that 
\begin{equation}\label{eq:rmk_1}
\left|G_{kj, \nu}^{(\mu)}(x_j)\right|\leq \left|G_k^{(\mu)}(x_k)\right|\cdot R^\nu\leq A_\mu R^\nu
\end{equation}
for each $\mu<m$, we obtain the inequality 
\begin{eqnarray}
|h_{1jk, m}|&\leq&
\text{the coefficient of}\ z_j^m\ \text{in the expansion of}\ \sum_{\mu=0}^{m-1}\sum_{\nu=1}^\infty A_\mu R^\nu\cdot z_j^{\nu+\mu} \nonumber \\
&=& \text{the coefficient of}\ z_j^m\ \text{in the expansion of}\ \left(\sum_{\mu=0}^{m-1} A_\mu z_j^\mu \right)\cdot\left(\sum_{\nu=1}^\infty R^\nu z_j^\nu \right) \nonumber \\
&\leq&\text{the coefficient of}\ X^m\ \text{in the expansion of}\ \frac{RXA(X)}{1-RX}. \nonumber 
\end{eqnarray}
Since 
\[
\left|g_{jk}^{(n+1, m)}\right|\leq MR^m=\text{the coefficient of}\ X^m\ \text{in the expansion of}\ \frac{M}{1-RX} 
\]
holds, it follows that
\begin{eqnarray}
\left|G_j^{(m)}\right| &\leq& K\cdot\max_{jk}\sup_{U_{jk}}\left|g_{jk}^{(n+1, m)}-t_{jk}^{-n}h_{1jk, m}\right| \nonumber \\
&\leq& \text{the coefficient of}\ X^m\ \text{in the expansion of}\ \frac{M+A(X)RX}{1-RX}. \nonumber
\end{eqnarray}
As $A$ is a solution of the functional equation
\[
\frac{K(M+A(X)RX)}{1-RX}=A(X), 
\]
the inequality $\left|G_j^{(m)}\right| \leq A_m$ holds. 
\end{proof}

\begin{remark}\label{rmk:1}
Strictly speaking, we have to enlarge $M$ and $R$ in the proof of Lemma \ref{lem:newsystem_1_torsion}, 
which is because we used the non-trivial assumption that 
\[
\{(x_j, z_j, w_j)\in W_j\mid x_j\in U_{jk}, |z_j|\leq R^{-1}, w_j=0\}\subset W_j\cap W_k
\]
by stealth in the proof of the equation $(\ref{eq:rmk_1})$. 
However, this difficulty can be avoided by using a new open covering $\{U_j^*\}$ of $C$ such that each $U_j^*$ is a relatively compact subset of $U_j$ (see \cite[p. 599]{U} for the details). After replacing $R$ with a sufficiently large number determined by this open covering $\{U_j^*\}$ and $M$ with $2M$, Lemma \ref{lem:newsystem_1_torsion} can be proven. 
\end{remark}

Let $\{G^{(m)}_j\}$ be that in Lemma \ref{lem:newsystem_1_torsion} and 
consider the function $F_j(x_j, z_j)=\sum_{m=0}^\infty G^{(m)}_j(x_j)\cdot z_j^m$. 
From Lemma \ref{lem:newsystem_1_torsion}, it can be said that $F_j$ is a holomorphic function around $C$. 
Let $\widetilde{w}_j$ be the solution of the functional equation $(\ref{eq:eqfunc1})$. 
Then, it can be showed by just the same argument as in the proof of Lemma \ref{lem:newsystem_1_torsion},  that $\{(W_j. \widetilde{w}_j)\}$ is a system of order $(n+1, 0)$, which shows Lemma \ref{lem:newsystem_2} when $(i')$ holds. 
\hfill $\Box$

\subsubsection{Proof of Lemma \ref{lem:newsystem_2} $(ii)$}

Here we prove Lemma \ref{lem:newsystem_2} in the case where the condition $(ii)$ $N_{C/S}=N_{S/X}|_C\in\mathcal{E}_1(C)$ holds. 
We consider the functional equation $(\ref{eq:eqfunc1})$ also in this case, 
and so what we have to do is determining each coefficients $\{G_j^{(\mu)}\}$ of $F_j$ 
just as in Lemma \ref{lem:newsystem_1_torsion}. 
However, in this case, the sequence of constants $\{K(N_{S/X}|_C^{-n}\otimes N_{C/S}^{-m})\}_{m=0}^\infty$ as in \cite[Lemma 3]{U} need not be bounded from above. 
So now we will use Lemma \ref{lem:ueda_4} instead of \cite[Lemma 3]{U}. 

Set $N:=N_{S/X}|_C=N_{C/S}$ and $\varepsilon_n:=d(\mathcal{O}_C, N^{-n})^{-1}$. 
Let $K$ be the constant as in Lemma \ref{lem:ueda_4}. 
By just the same argument as in the proof of Lemma \ref{lem:newsystem_1_torsion}, 
we can inductively define the functions $G_j^{(m)}$ such that $\delta\{(U_j, G_j^{(m)})\}=\{(U_{jk}, g_{jk}^{(n+1, m)}-t_{jk}^{-n}h_{1jk, m})\}$ and 
\[
\max_j\sup_{U_j}\left|G_j^{(m)}\right|\leq \varepsilon_{n+m}K\cdot\max_{jk}\sup_{U_{jk}}\left|g_{jk}^{(n+1, m)}-t_{jk}^{-n}h_{1jk, m}\right|
\]
hold for each $m$ and $j, k$. 
Thus all we have to do is showing that the function $F_j=\sum_{m=0}^\infty G^{(m)}_{j}z_j^m$ is actually a holomorphic function around $C$. 
For this purpose, we will prove that there exists a power series $B(X)=\sum_{m=0}^\infty B_mX^m$ with a positive radius of convergence such that 
$B_0=KM$ and 
\begin{equation}\label{eq:koutousiki}
\sum_{\mu=1}^\infty \varepsilon_{n+\mu}^{-1}B_\mu X^\mu=\frac{K(M+B(X))RX}{1-RX}
\end{equation}
holds, where the constants $M$ and $R$ are that in the proof of Lemma \ref{lem:newsystem_2} when $(i')$ holds. 
Note that the power series $B$ is uniquely determined as a formal power series, since all of the coefficients $B_m$ are inductively determined by the above equation $(\ref{eq:koutousiki})$. 
By using the equation $(\ref{eq:koutousiki})$, it is also shown by the induction that each coefficient $B_m$ is non-negative real number. 
From now on, we will prove that this formal power series $B(X)$ has a positive radius of convergence. 
Consider a formal power series $D(X)=X+\sum_{\lambda=2}^\infty D_\lambda X^\lambda$ defined by 
\[
\sum_{\lambda=2}^\infty \varepsilon_{\lambda-1}^{-1}D_\mu X^\lambda=\frac{KRD(X)^2(1+MD(X)^{n})}{1-RD(X)}. 
\]
Just as the above argument on $B(X)$, this power series $D(X)$ is also uniquely determined as a formal power series with non-negative real coefficients. 

\begin{claim}\label{claim:tilde_B}
$B_\mu\leq D_{\mu+n+1}$ holds for each $\mu\geq 0$. 
\end{claim}

\begin{proof}
Consider the power series $\widetilde{B}(X):=X+B(X)\cdot X^{n+1}$. 
Denote by $\widetilde{B}_\lambda$ the coefficient of $X^\lambda$ in the expansion of $\widetilde{B}(X)$. 
We will prove the inequality $\widetilde{B}_\lambda\leq D_\lambda$ for each $\lambda$ by induction. 
First, it is clear that this inequality holds for $\lambda=1, 2, \dots, n$. 
Next, let us assume that $\widetilde{B}_\lambda\leq D_\lambda$ holds for $\lambda<\ell$ and prove that $\widetilde{B}_\ell\leq D_\ell$. 
It follows from the equation $(\ref{eq:koutousiki})$ that 
\[
\sum_{\lambda=1}^\infty \varepsilon^{-1}_{\lambda-1}\widetilde{B}_\lambda X^{\lambda}=\frac{K(MX^{n+1}+\widetilde{B}(X)-X)RX}{1-RX}. 
\]
Thus all we have to do is showing 
\begin{eqnarray}
&&\text{the coefficient of }X^\ell\ \text{in the expansion of}\ 
\frac{K(MX^{n+1}+\widetilde{B}(X)-X)RX}{1-RX}\nonumber \\
&\leq&
\text{the coefficient of }X^\ell\ \text{in the expansion of}\ 
\frac{KRD(X)^2(1+MD(X)^{n})}{1-RD(X)}. \nonumber 
\end{eqnarray}
Since the left (resp. right) hand side of the above inequality depends only on $\{\widetilde{B}_\lambda\}_{\lambda<\ell}$ (resp. $\{D_\lambda\}_{\lambda<\ell}$), it follows from the assumption of the induction that it is sufficient to show the inequality 
\begin{eqnarray}
&&\text{the coefficient of }X^\ell\ \text{in the expansion of}\ \frac{K(MX^{n+1}+\widetilde{B}(X)-X)RX}{1-RX}\nonumber \\
&\leq&\text{the coefficient of }X^\ell\ \text{in the expansion of}\ 
\frac{KR\widetilde{B}(X)^2(1+M\widetilde{B}(X)^{n})}{1-R\widetilde{B}(X)}, \nonumber
\end{eqnarray}
which is easily obtained from the fact that the coefficients of $\widetilde{B}(X)-X$ and $X$ are less than or equal to that of $\widetilde{B}(X)$. 
\end{proof}

According to Claim \ref{claim:tilde_B}, it is sufficient for proving Lemma
\ref{lem:newsystem_2} to show that the formal power series ${D}(X)$ has a positive radius of convergence. 
We show this by using Lemma \ref{lem:siegel}. 
Note that our $\{\varepsilon_\lambda\}_{\lambda\geq 1}$ enjoys the two conditions in Lemma \ref{lem:siegel} (here we used the assumption $(ii)$, see also \cite[p. 603]{U}). Thus we can apply Lemma \ref{lem:siegel} on 
\[
f(Z):=\frac{KRZ^2(1+MZ^{n})}{1-RZ}
\]
and conclude that $D(X)$ has a positive radius of convergence. 
Thus $B(X)$ and $F_j$ also have a positive radius of convergence. 
Therefore we can construct a system of order $(n+1, 0)$ by solving the functional equation $(\ref{eq:eqfunc1})$, which completes the proof of Lemma \ref{lem:newsystem_2} in the case where the condition $(ii)$ holds. 
\hfill $\Box$

\subsubsection{Proof of Lemma \ref{lem:newsystem_2} $(iii')$}
Finally we prove Lemma \ref{lem:newsystem_2} in the case where the condition $(iii')$ holds. 
In this case, we can apply \ref{prop:Rossi} and thus we may assume that the neighborhood $V$ satisfies the following property: there exists an integer $m$ such that the natural map $H^1(V, N_{S/X}|_V^{-n})\to H^1(V, N_{S/X}|_V^{-n}\otimes\mathcal{O}_{V}/I_C^m)$ is injective, where $I_C$ is the defining ideal sheaf of $C\subset V$. 
Fix a system $\{(W_j, w_j)\}$ of order $(n, m+1)$. 
It clearly follows from the equation $(\ref{eq:mero_exp})$ that the cohomology class 
$[\{(V_j, (w_j^{-n}-t_{jk}^{-n}w_k^{-n})|_{V_{jk}})\}]_m\in H^1(V, N_{S/X}|_V^{-n}\otimes\mathcal{O}_{V}/I_C^m)$ induced from $[\{(V_j, (w_j^{-n}-t_{jk}^{-n}w_k^{-n})|_{V_{jk}})\}]\in H^1(V, N_{S/X}|_V^{-n})$ is the trivial one. 
Thus we can conclude that the cohomology class 
$[\{(V_j, (w_j^{-n}-t_{jk}^{-n}w_k^{-n})|_{V_{jk}})\}]\in H^1(V, N_{S/X}|_V^{-n})$ itself is also the trivial one. 
Therefore there exists a system $\{(V_{j}, F_j(x_j, z_j))\}$ of holomorphic functions such that 
 \[
   (w_j^{-n}-t_{jk}^{-n}w_k^{-n})|_{V_{jk}}=F_j-t_{jk}^{-n}F_k
 \]
holds on each $V_{jk}$. 
Define 
 \[
  \widetilde{w}_j:=w_j\cdot (1-F_j\cdot w_j^n)^{-\frac{1}{n}}
 \]
by shrinking $V$ and $W$ if necessary. 
Since
 \[
  \widetilde{w}_j=w_j\cdot \left(1+\frac{1}{n}F_j\cdot w_j^n+\cdots\right)
  =w_j+O(w_j^{n+2})
 \]
holds, it is clear that our new system $\{\widetilde{w}_j\}$ is also of order $(n, 0)$. 
As $\widetilde{w}_j^{-n}=w_j^{-n}-F_j$ 
holds, we obtain that
 \[
  \frac{1}{n}(\widetilde{w}_j^{-n}-t_{jk}^{-n}\widetilde{w}_k^{-n})|_{V_{jk}}\equiv 0, 
 \]
which means that the system $\{\widetilde{w}_j\}$ is of order $(n+1, 0)$. 
\hfill $\Box$

\section{Proof of Theorem \ref{main}}

In this section we prove Theorem \ref{main}. 
The strategy of the proof is almost the same as that of Lemma \ref{lem:newsystem_2}: 
i.e. fixing a system $\{(W_j, w_j)\}$ of order $(1, 0)$, 
we will construct a new system $\{(W_j, \widetilde{w}_j)\}$ by solving the functional equation 
\begin{equation}\label{eq:eqfunc11}
w_j=u_j+\sum_{\nu=2}^\infty F_j^{(\nu)}(x_j, z_j)\cdot u_j^{\nu}
\end{equation}
after choosing a suitable system of holomorphic functions $\{(V_{j}, F^{(\nu)}_j(x_j, z_j))\}$. 
Let
\[
F_j^{(\nu)}(x_j, z_j)=\sum_{\mu=0}G_j^{(\nu, \mu)}(x_j)\cdot z_j^\mu
\]
be the expansion of $F^{(\nu)}_j(x_j, z_j)$. 
\hfill $\Box$

\subsection{Proof of Theorem \ref{main} $(i)$}

First we prove Theorem \ref{main} in the case where the condition $(i)$ $N_{C/S}, N_{S/X}|_C\in\mathcal{E}_0(C)$ holds. 
Let $K_{n, m}=K(N_{S/X}|_C^{-n}\otimes N_{C/S}^{-m})$ be the constant as in \cite[Lemma 3]{U} and 
set $K:=\max_{n, m} K_{n, m}$ (here the condition $(i)$ is needed). 
Let 
\begin{equation}\label{eq:thm_tenkai}
t_{jk}w_k-w_j=\sum_{\nu=2}^\infty f_{jk}^{(\nu)}(x_j, z_j)\cdot w_j^{\nu}=\sum_{\nu=2}^\infty\sum_{\mu=0}^\infty g_{jk}^{(\nu, \mu)}(x_j)\cdot z_j^\mu\cdot w_j^{\nu}
\end{equation}
be the expansions of $t_{jk}w_k-w_j$. 
Let $M$ and $R$ be sufficiently large positive number such that $\sup_{W_{jk}}|t_{jk}w_k-w_j|<M$ and $\{|z_j|<2R^{-1}, |w_j|<2R^{-1}\}\subset W_{j}$ hold for each $j, k$. 
Note that, from Cauchy's inequality, it holds that $|g^{(n, m)}_{jk}|<MR^{n+m+1}$. 
Consider the solution $A(X, Y)$ of the functional equation 
 \begin{equation}\label{eq:funceq_A}
  A(X, Y)-X=\frac{RK}{1-RY}\left((A(X, Y)-X)Y+\frac{(1+MR)A(X, Y)^2}{1-RA(X, Y)}\right)
 \end{equation}
 and denote by $A_{n, m}$ the coefficient of $X^nY^m$ in the Taylor expansion of $A(X)$ at $X=Y=0$. 
Though the functional equation $(\ref{eq:funceq_A})$ has two solutions, the solution $A$ is uniquely determined by the condition that $A(X, Y)=X+O(X^2)$. 
Note that $A_{n, m}$ is a non-negative real number for each $n$ and $m$. 

\begin{lemma}\label{lem:newsystem_2_torsion}
There exists a system of functions $\{G_j^{(n, m)}(x_j)\}_{n\geq 2, m\geq 0}$ for each $j$ satisfying the following conditions. 
Let 
\begin{eqnarray}
G_j^{(n, m)}(x_j)\cdot z_j^m&=&G_j^{(n, m)}(x_j(x_k, z_k, w_k))\cdot z_j(x_k, z_k, w_k)^m \nonumber \\
&=&G_j^{(n, m)}(x_j(x_k, 0, 0))\cdot s_{jk}^m z_k^m\nonumber \\
&&+\sum_{q=1}^\infty G_{jk, 0, q}^{(n, m)}(x_k)\cdot z_k^{m+q}
+\sum_{p=1}^\infty\sum_{q=0}^\infty G_{jk, p, q}^{(n, m)}(x_k)\cdot z_k^{m+q}w_k^{p}\nonumber
\end{eqnarray}
be the expansion of $(G_j^{(n, m)}z_j^m)|_{W_{jk}}$ in the variable $z_k$ and $w_k$ by regarding $G_j^{(n, m)}z_j^m$ as a function defined on $W_{j}$ which does not depend on $w_j$. 
Denote by $h_{1jk, n, m}$ the coefficient of $z_j^m$ in the expansion of 
\[
\sum_{\mu=0}^{m-1}\sum_{q=1}^\infty  G_{kj, 0, q}^{(n, \mu)}(x_j)\cdot z_j^{\mu+q}, 
\]
by $h_{2jk, n, m}$ the coefficient of $z_j^mu_j^n$ in the expansion of 
\[
\sum_{\nu=2}^{n-1}\sum_{\mu=0}^m\sum_{p=1}^\infty\sum_{q=0}^\infty G_{kj, p, q}^{(\nu, \mu)}(x_j)\cdot z_j^{\mu+q}u_j^\nu\cdot\left(u_j+\sum_{a=2}^{n-1}\sum_{b=0}^mG_j^{(a, b)}(x_j)\cdot z_j^bu_j^a\right)^{p}, 
\]
and by $h_{3jk, n, m}$ the coefficient of $z_j^mu_j^n$ in the expansion of 
\[
\sum_{\nu=2}^{\infty}\sum_{\mu=0}^{\infty}g_{jk}^{(\nu, \mu)}(x_j)\cdot z_j^\mu\cdot\left(u_j+\sum_{p=2}^{n-1}\sum_{q=0}^{m}G_j^{(p, q)}(x_j)\cdot z_j^qu_j^p\right)^\nu
-\sum_{\nu=2}^{n}\sum_{\mu=0}^{m}g_{jk}^{(\nu, \mu)}(x_j)\cdot z_j^\mu u_j^\nu. 
\]
 Then the coboundary $\delta\{(U_j, G_j^{(n, m)})\}$ is equal to 
\[
\{(U_{jk}, g_{jk}^{(n, m)}-t_{jk}^{-n+1}h_{1jk, n, m}-t_{jk}^{-n+1}h_{2jk, n, m}+h_{3jk, n, m})\}, 
\]
 and 
$\max_j\sup_{U_j}\left|G_j^{(n, m)}\right|\leq A_{n, m}$ for each $n\geq 1, m\geq 0$. 
\end{lemma}

\begin{proof}
It immediately follows from the definition that $\{(U_{jk}, g_{jk}^{(2, 0)})\}$ is a $1$-cocycle which defines the cohomology class $u_{1, 0}(C, S, X)$, which is equal to $0\in H^1(C, N_{S/X}|_C^{-1})$ from the assumption. 
Thus there exists an $1$-cochain $\{(U_j, G_j^{(2, 0)})\}$ of $N_{S/X}|_C^{-n}$ such that 
\[
\delta\{(U_j, G_j^{(2, 0)})\}=\{(U_{jk}, g_{jk}^{(2, 0)})\},\ 
\max_j\sup_{U_j}\left|G_j^{(2, 0)}\right|\leq K\cdot\max_{jk}\sup_{U_{jk}}\left|g_{jk}^{(2, 0)}\right|
\]
hold. As $h_{1jk, 2, 0}=h_{2jk, 2, 0}=h_{3jk, 2, 0}=0$ and $K\cdot |g_{jk}^{(2, 0)}|\leq KMR^2\leq RK(1+MR)=A_{2, 0}$, $\{(G_j^{(2, 0)})\}$ is what we wanted. 

Fix $(n, m)$ and assume that there exists a system 
$\{G_j^{(\nu, \mu)}\}$ as in Lemma \ref{lem:newsystem_1_torsion} for each $(\nu, \mu)$ less than $(n, m)$ in the lexicographical order. 
From now on, we will construct $\{G_j^{(n, m)}\}$. 
For simplicity, we assume that $m>0$ (the construction of $\{G_j^{(n+1, 0)}\}$ is just the same as the following construction formally replaced $(n, m)$ with $(n, \infty)$ ). 
Denote by $\widetilde{w}_j$ the solution of the functional equation 
\begin{equation}\label{eq:eqfunc12}
w_j=u_j+\sum_{\nu=2}^{n-1}F_j^{(\nu)}(x_j, z_j)\cdot u_j^\nu+\left(\sum_{\mu=0}^{m-1} G_j^{(n, \mu)}(x_j)\cdot z_j^\mu\right)\cdot u_j^{n}
\end{equation}
with an unknown function $u_j$. 

First we will show that our new system $\{(W_j, \widetilde{w}_j)\}$ is of order $(n-1, 0)$. 
Note that $\{(W_j, \widetilde{w}_j)\}$ is clearly of order $(1, 0)$. 
Thus all we have to do is to show that $\{(W_j, \widetilde{w}_j)\}$ is of order $(a-1, 0)$ assuming that $\{(W_j, \widetilde{w}_j)\}$ is of order $(a-2, 0)$ for each $a\leq n$. 

From the functional equation $(\ref{eq:eqfunc12})$, the function $(t_{jk}w_k)|_{W_{jk}}$ can be expanded as follows: 
\begin{eqnarray}
& &t_{jk}w_k \nonumber \\
&=& t_{jk}\widetilde{w}_k+\sum_{\nu=2}^{a-1}\sum_{\mu=0}^{\infty} G_k^{(\nu, \mu)}(x_k)\cdot z_k^\mu\cdot t_{jk}\widetilde{w}_k^\nu+O(\widetilde{w}_k^a) \nonumber \\
&=& t_{jk}\widetilde{w}_k+\sum_{\nu=2}^{a-1}\sum_{\mu=0}^{\infty}G_k^{(\nu, \mu)}(x_k(x_j, 0, 0))\cdot t_{jk}s_{jk}^{-\mu}z_j^{-\mu}\widetilde{w}_k^\nu 
+\sum_{\nu=2}^{a-1}\sum_{\mu=0}^\infty \sum_{q=1}^\infty  G_{kj, 0, q}^{(n, \mu)}(x_j)\cdot z_j^{\mu+q}t_{jk}\widetilde{w}_k^\nu \nonumber \\
&&+\sum_{\nu=2}^{a-1}\sum_{\mu=0}^{\infty}\sum_{p=1}^\infty\sum_{q=0}^\infty G_{kj, p, q}^{(\nu, \mu)}(x_j)\cdot z_j^{\mu+q}w_j^p\cdot t_{jk}\widetilde{w}_k^\nu +O(\widetilde{w}_k^a)\nonumber \\
&=& t_{jk}\widetilde{w}_k+\sum_{\nu=2}^{a-1}\sum_{\mu=0}^{\infty}G_k^{(\nu, \mu)}(x_k(x_j, 0, 0))\cdot t_{jk}^{-\nu+1}s_{jk}^{-\mu}z_j^{-\mu}\widetilde{w}_j^\nu 
+\sum_{\nu=2}^{a-1}\sum_{\mu=0}^\infty \sum_{q=1}^\infty  G_{kj, 0, q}^{(n, \mu)}(x_j)\cdot z_j^{\mu+q}t_{jk}\widetilde{w}_k^\nu \nonumber \\
&&+\sum_{\nu=2}^{a-1}\sum_{\mu=0}^{\infty}\sum_{p=1}^\infty\sum_{q=0}^\infty G_{kj, p, q}^{(\nu, \mu)}(x_j)\cdot z_j^{\mu+q}w_j^p\cdot t_{jk}^{-\nu+1}\widetilde{w}_j^\nu +O(\widetilde{w}_k^a)\nonumber \\
&=& t_{jk}\widetilde{w}_k+\sum_{\nu=2}^{a-1}\sum_{\mu=0}^{\infty}G_k^{(\nu, \mu)}(x_k(x_j, 0, 0))\cdot t_{jk}^{-\nu+1}s_{jk}^{-\mu}z_j^{-\mu}\widetilde{w}_j^\nu 
+\sum_{\nu=2}^{a-1}\sum_{\mu=0}^\infty h_{1jk, \nu, \mu}t_{jk}^{-\nu+1}z_j^\mu\widetilde{w}_j^\nu \nonumber \\
&&+\sum_{\nu=0}^{a-1}\sum_{\mu=0}^\infty t_{jk}^{-\nu+1}h_{2jk, \nu, \mu}z_j^\mu  \widetilde{w}_j^\nu+O(\widetilde{w}_k^a).  \nonumber 
\end{eqnarray}
Here we used the fact that $h_{1jk, \nu, \mu}$ and $h_{2jk, \nu, \mu}$ depend only on $\{G^{(p, q)}_j\}_{(p, q)<(\nu, \mu)}$. 
Note that we also used the fact that $\widetilde{w}_k^\nu=t_{jk}^{-\nu}\widetilde{w}_j^\nu+O(\widetilde{w}_j^a)$ holds for each $\nu\geq 2$, which directly follows from the assumption that $\{(W_j, \widetilde{w}_j)\}$ is of order $(a-2, 0)$. 
On the other hand, the function $(t_{jk}w_k)|_{W_{jk}}$ can also be expanded as follows: 
\begin{eqnarray}\label{eq:tw_tenkai_12}
& &t_{jk}w_k  \\
&=& w_j+\sum_{\nu=2}^{n-1} f_{jk}^{(\nu)}(x_j, z_j)\cdot w_j^{\nu}+\left(\sum_{\mu=0}^m g_{jk}^{(n, \mu)}(x_j)\cdot z_j^\mu+O(z_j^{m+1})\right)\cdot w_j^{n}+O(w_j^{n+1}) \nonumber 
\end{eqnarray}
\begin{eqnarray}
&=& w_j+\sum_{\nu=2}^{n-1} f_{jk}^{(\nu)}(x_j, z_j)\cdot \left(\widetilde{w}_j+\sum_{\nu=2}^{n-1}F_j^{(\nu)}(x_j, z_j)\cdot \widetilde{w}_j^\nu+\left(\sum_{\mu=0}^{m-1} G_j^{(n, \mu)}(x_j)\cdot z_j^\mu\right)\cdot \widetilde{w}_j^{n}\right)^{\nu} \nonumber \\
&&+\left(\sum_{\mu=0}^m g_{jk}^{(n, \mu)}(x_j)\cdot z_j^\mu+O(z_j^{m+1})\right)\cdot \widetilde{w}_j^{n}+O(\widetilde{w}_j^{n+1}) \nonumber \\
&=& w_j+\sum_{\nu=2}^{n-1}\left(f_{jk}^{(\nu)}(x_j, z_j)+\sum_{\mu=0}^\infty h_{3jk, \nu, \mu} z_j^\mu\right)\cdot \widetilde{w}_j^\nu\nonumber \\
&& +\left(\sum_{\mu=0}^{m} h_{3jk, n, \mu} z_j^\mu+\sum_{\mu=0}^m g_{jk}^{(n, \mu)}(x_j)\cdot z_j^\mu+O(z_j^{m+1})\right)\cdot \widetilde{w}_j^{n}+O(\widetilde{w}_j^{n+1}) \nonumber \\
&=& \widetilde{w}_j+\sum_{\nu=2}^{n-1}\left(F_j^{(\nu)}(x_j, z_j)+f_{jk}^{(\nu)}(x_j, z_j)+\sum_{\mu=0}^\infty h_{3jk, \nu, \mu} z_j^\mu\right)\cdot \widetilde{w}_j^\nu \nonumber \\
&&+\left(\sum_{\mu=0}^{m-1} \left(G_{j}^{(n, \mu)}(x_j)+g_{jk}^{(n, \mu)}(x_j)+h_{3jk, n, \mu}\right)\cdot z_j^\mu\right)\cdot \widetilde{w}_j^{n} \nonumber \\
&&+\left(\left(g_{jk}^{(n, m)}(x_j)+h_{3jk, n, m}\right)\cdot z_j^m+O(z_j^{m+1})\right)\cdot \widetilde{w}_j^{n}+O(\widetilde{w}_j^{n+1}). \nonumber  
\end{eqnarray}
Here we used the fact that $h_{3jk, \nu, \mu}$ depends only on $\{G^{(p, q)}_j\}_{(p, q)<(\nu, \mu)}$. 
From these two expansions, it follows that 
\begin{eqnarray}
t_{jk}\widetilde{w}_k-\widetilde{w}_j&=&\sum_{\nu=2}^{a-1}\sum_{\mu=0}^{\infty} \left(-t_{jk}^{-\nu+1}s_{jk}^{-\mu}G_k^{(\nu, \mu)}(x_k(x_j, 0, 0))-t_{jk}^{-\nu+1}h_{1jk, \nu, \mu}-t_{jk}^{-\nu+1}h_{2jk, \nu, \mu}\right.\nonumber \\
&&\hskip15mm\left.+G_{j}^{(\nu, \mu)}(x_j)+g_{jk}^{(\nu, \mu)}(x_j)+h_{3jk, \nu, \mu}\right)\cdot z_j^\mu\widetilde{w}_j^\nu +O(\widetilde{w}_j^{a}). \nonumber 
\end{eqnarray}
As it follows from the assumption of the induction that the coboundary $\delta\{(U_j, G_j^{(\nu, \mu)})\}$ is equal to 
\[
\{(U_{jk}, g_{jk}^{(\nu, \mu)}-t_{jk}^{-\nu+1}h_{1jk, \nu, \mu}-t_{jk}^{-\nu+1}h_{2jk, \nu, \mu}+h_{3jk, \nu, \mu})\}, 
\]
for $\nu\leq a (<n)$ and $\nu\geq 0$, 
we can conclude that $t_{jk}\widetilde{w}_k-\widetilde{w}_j=O(\widetilde{w}_j^{a}) $, which means that the system $\{\widetilde{w}_j\}$ is of order $(a-1, 0)$. 
Therefore, it can be said that the system $\{\widetilde{w}_j\}$ is of order $(n-1, 0)$. 

Using this fact and the functional equation $(\ref{eq:eqfunc12})$, let us consider again the expansion of function $(t_{jk}w_k)|_{W_{jk}}$: 

\begin{eqnarray}
& &t_{jk}w_k \nonumber \\
&=& t_{jk}\widetilde{w}_k+\sum_{\nu=2}^{n-1}\sum_{\mu=0}^{\infty} G_k^{(\nu, \mu)}(x_k)\cdot z_k^\mu\cdot t_{jk}\widetilde{w}_k^\nu+\left(\sum_{\mu=0}^{m-1} G_k^{(n, \mu)}(x_k)\cdot z_k^\mu\right)\cdot t_{jk}\widetilde{w}_k^{n} \nonumber 
\end{eqnarray}
\begin{eqnarray}
&=& t_{jk}\widetilde{w}_k+\sum_{\nu=2}^{n-1}\sum_{\mu=0}^{\infty}G_k^{(\nu, \mu)}(x_k(x_j, 0, 0))\cdot t_{jk}s_{jk}^{-\mu}z_j^{-\mu}\widetilde{w}_k^\nu 
+\sum_{\nu=2}^{n-1}\sum_{\mu=0}^\infty h_{1jk, \nu, \mu}t_{jk}z_j^\mu\widetilde{w}_k^\nu \nonumber \\
&&+\sum_{\nu=2}^{n-1}\sum_{\mu=0}^{\infty}\sum_{p=1}^\infty\sum_{q=0}^\infty G_{kj, p, q}^{(\nu, \mu)}(x_j)\cdot z_j^{\mu+q}w_j^p\cdot t_{jk}\widetilde{w}_k^\nu \nonumber \\
 &&+\left(\sum_{\mu=0}^{m-1}\left( G_k^{(n, \mu)}(x_k(x_j, 0, 0))\cdot s_{jk}^{-\mu}+h_{1jk, n, \mu}z_j^\mu\right)\cdot z_j^{\mu}\right)\cdot t_{jk}\widetilde{w}_k^{n}\nonumber \\
&&+\left( h_{1jk, n, m}z_j^m+O(z_j^{m+1})\right)\cdot t_{jk}\widetilde{w}_k^{n}+O(\widetilde{w}_k^{n+1}) \nonumber \\
&=& t_{jk}\widetilde{w}_j+\sum_{\nu=2}^{n-1}\sum_{\mu=0}^{\infty}G_k^{(\nu, \mu)}(x_k(x_j, 0, 0))\cdot t_{jk}^{-\nu+1}s_{jk}^{-\mu}z_j^{-\mu}\widetilde{w}_j^\nu 
+\sum_{\nu=2}^{n-1}\sum_{\mu=0}^\infty h_{1jk, \nu, \mu}t_{jk}^{-\nu+1}z_j^\mu\widetilde{w}_j^\nu \nonumber \\
&&+\sum_{\nu=2}^{n-1}\sum_{\mu=0}^{\infty}\sum_{p=1}^\infty\sum_{q=0}^\infty G_{kj, p, q}^{(\nu, \mu)}(x_j)\cdot z_j^{\mu+q}w_j^p\cdot t_{jk}^{-\nu+1}\widetilde{w}_j^\nu \nonumber \\
 &&+\left(\sum_{\mu=0}^{m-1}\left( G_k^{(n, \mu)}(x_k(x_j, 0, 0))\cdot s_{jk}^{-\mu}+h_{1jk, n, \mu}z_j^\mu\right)\cdot z_j^{\mu}\right)\cdot t_{jk}^{-\nu+1}\widetilde{w}_j^{n} \nonumber \\
&&+\left( h_{1jk, n, m}z_j^m+O(z_j^{m+1})\right)\cdot t_{jk}^{-\nu+1}\widetilde{w}_j^{n}+O(\widetilde{w}_j^{n+1}) \nonumber \\
&=& t_{jk}\widetilde{w}_j+\sum_{\nu=2}^{n-1}\sum_{\mu=0}^{\infty}G_k^{(\nu, \mu)}(x_k(x_j, 0, 0))\cdot t_{jk}^{-\nu+1}s_{jk}^{-\mu}z_j^{-\mu}\widetilde{w}_j^\nu 
+\sum_{\nu=2}^{n-1}\sum_{\mu=0}^\infty h_{1jk, \nu, \mu}t_{jk}^{-\nu+1}z_j^\mu\widetilde{w}_j^\nu \nonumber \\
&&+\sum_{\nu=0}^{n}\left(\sum_{\mu=0}^\infty t_{jk}^{-\nu+1}h_{2jk, \nu, \mu}z_j^\mu \right)\cdot \widetilde{w}_j^\nu \nonumber \\
 &&+\left(\sum_{\mu=0}^{m-1}\left( G_k^{(n, \mu)}(x_k(x_j, 0, 0))\cdot s_{jk}^{-\mu}+h_{1jk, n, \mu}z_j^\mu\right)\cdot z_j^{\mu}\right)\cdot t_{jk}^{-\nu+1}\widetilde{w}_j^{n} \nonumber \\
&& +\left( h_{1jk, n, m}z_j^m+O(z_j^{m+1})\right)\cdot t_{jk}^{-\nu+1}\widetilde{w}_j^{n}+O(\widetilde{w}_j^{n+1}).  \nonumber 
\end{eqnarray}
Comparing this expansion with the previous expansion $(\ref{eq:tw_tenkai_12})$, we obtain that 
\begin{eqnarray}
t_{jk}\widetilde{w}_k-\widetilde{w}_j
&=&\left(-t_{jk}^{-n+1}h_{1jk, n, m}-t_{jk}^{-n+1}h_{2jk, n, m}+g_{jk}^{(n, m)}(x_j)+h_{3jk, n, m}\right)\cdot z_j^m\widetilde{w}_j^{n}\nonumber \\
&&+O(z_j^{m+1})\cdot \widetilde{w}_j^{n}+O(\widetilde{w}_j^{n+1}).  \nonumber 
\end{eqnarray}
Note that here we again used the assumption of the induction on the coboundary \linebreak$\delta\{(U_j, G_j^{(\nu, \mu)})\}$. 
Now it can be said that the system 
\[
\{(U_{jk}, g_{jk}^{(n, m)}(x_j)-t_{jk}^{-n+1}h_{1jk, n, m}-t_{jk}^{-n+1}h_{2jk, n, m}+h_{3jk, n, m})\}
\]
 is an $1$-cocycle which defines the $(n-1, m)$-th Ueda class, which is the trivial one from the assumption of Theorem \ref{main}. 
Therefore it follows from \cite[Lemma 3]{U} that there exists a $0$-cochain $\{(U_j, G^{(n, m)}_j)\}$ such that $\delta\{(U_j, G^{(n, m)}_j)\}$ coincides with the above $1$-cocycle and 
\[
\max_j\sup_{U_j}\left|G_j^{(n, m)}\right|\leq K\cdot \max_{jk}\sup_{U_{jk}}\left|g_{jk}^{(n, m)}(x_j)-t_{jk}^{-n+1}h_{1jk, n, m}-t_{jk}^{-n+1}h_{2jk, n, m}+h_{3jk, n, m}\right| 
\]
holds. 
From the definition of $h_{1jk, n, m}$, $h_{2jk, n, m}$, and $h_{3jk, n, m}$ (and the assumption of the induction), we obtain the following inequalities: 
\begin{eqnarray}
 |h_{1jk, n, m}|
 &\leq&\text{the coefficient of }z_j^m\ \text{in the expansion of }\sum_{\mu=0}^{m-1}\sum_{q=1}^\infty  A_{n, \mu}R^q\cdot z_j^{\mu+q}\nonumber \\
 &\leq&\text{the coefficient of }X^nY^m\ \text{in the expansion of }\frac{RY(A(X, Y)-X)}{1-RY}, \nonumber 
\end{eqnarray}

\begin{eqnarray}
 |h_{2jk, n, m}|&\leq& 
\text{the coefficient of }u_j^nz_j^m \nonumber \\
&&\hskip10mm \text{in the expansion of }\sum_{\nu=2}^{n-1}\sum_{\mu=0}^m\sum_{p=1}^\infty\sum_{q=0}^\infty A_{\nu, \mu}R^{p+q} z_j^{\mu+q}u_j^\nu\cdot A(u_j, z_j)^{p}\nonumber \\
&\leq&\text{the coefficient of }X^nY^m \nonumber \\
&&\hskip10mm \text{in the expansion of }A(X, Y)\sum_{p=1}^\infty R^pA(X, Y)^p\sum_{q=0}^\infty R^qY^q\nonumber \\
 &=&\text{the coefficient of }X^nY^m\text{in the expansion of }\frac{RA(X, Y)^2}{(1-RY)(1-RA(X, Y))}, \nonumber 
\end{eqnarray}

\begin{eqnarray}
 |g^{(n, m)}_{jk}+h_{3jk, n, m}|&\leq&
\text{the coefficient of }u_j^nz_j^m \nonumber \\
&&\hskip10mm \text{in the expansion of }\sum_{\nu=2}^{\infty}\sum_{\mu=0}^{\infty}MR^{\nu+\mu}z_j^\mu A(u_j, z_j)^\nu\nonumber \\
 &\leq&\text{the coefficient of }X^nY^m \nonumber \\
&&\hskip10mm \text{in the expansion of }M\sum_{\nu=2}^{\infty}R^\nu A(X, Y)^\nu\sum_{\mu=0}^{\infty}R^{\mu}Y^\mu\nonumber \\
 &=&\text{the coefficient of }X^nY^m \nonumber \\
&&\hskip10mm\text{in the expansion of } \frac{MR^2A(X, Y)^2}{(1-RY)(1-RA(X, Y))}. \nonumber 
\end{eqnarray}

Note that, strictly speaking, here we have to modify the constants $R$ and $M$ (see Remark \ref{rmk:1}). 
From the above three inequalities, it follows that $\max_j\sup_{U_j}|G_j^{(n, m)}|$ is less than or equal to the coefficient of $X^nY^m$ in the expansion of 
\[
  \frac{RK}{1-RY}\left((A(X, Y)-X)Y+\frac{(1+MR)A(X, Y)^2}{1-RA(X, Y)}\right), 
\]
which is equal to $A_{n, m}$ from the functional equation $(\ref{eq:funceq_A})$. 
\end{proof}

Let $\{G^{(n, m)}_j\}$ be that in Lemma \ref{lem:newsystem_2_torsion} and 
consider the function $F^{(n)}_j=\sum_{m=0}^\infty G^{(n, m)}_jz_j^m$. 
From Lemma \ref{lem:newsystem_2_torsion}, it can be said that $\sum_{n=2}^\infty F^{(n)}_ju_j^n$ is a holomorphic function around $C$. 
Let $\widetilde{w}_j$ be the solution of the functional equation $(\ref{eq:eqfunc11})$. 
Then, it follows from the same argument as in the proof of Lemma \ref{lem:newsystem_2_torsion} that $t_{jk}\widetilde{w}_k-\widetilde{w}_j\equiv 0$ holds for each $j$ and $k$, which shows Theorem \ref{main} when $(i)$ holds. 

\subsection{Proof of Theorem \ref{main} $(ii)$}

Next we prove Theorem \ref{main} in the case where the condition $(ii)$ holds. 
By considering the proof when $(i)$ holds with replacing the functional equation $(\ref{eq:funceq_A})$ with 
\begin{equation}\label{eq:funceq_A_2}
\sum_{\nu=2}^\infty\sum_{\mu=0}^\infty\varepsilon_{\nu+\mu-1}^{-1}A_{\nu, \mu}X^\nu Y^\mu=\frac{RK}{1-RY}\left((A(X, Y)-X)Y+\frac{(1+MR)A(X, Y)^2}{1-RA(X, Y)}\right)
 \end{equation}
and using Lemma \ref{lem:ueda_4} instead of  \cite[Lemma 3]{U}, 
it can be said that the proof of Theorem \ref{main} in this case is reduced to show that the formal power series solution $A(X, Y)=X+O(X^2)$ of the functional equation $(\ref{eq:funceq_A_2})$ is a holomorphic function around $X=Y=0$.  
Note that, by enlarging $K$, we can replace the functional equation $(\ref{eq:funceq_A_2})$ with
  \begin{equation}\label{eq:funceq_A_3}
 \sum_{\nu=2}^\infty\sum_{\mu=0}^\infty\varepsilon_{\nu+\mu-1}^{-1}A_{\nu, \mu}X^\nu Y^\mu=\frac{K}{1-RY}\left((A(X, Y)-X)Y+\frac{A(X, Y)^2}{1-RA(X, Y)}\right). 
 \end{equation}
From now on, we will prove that the formal solution $A(X, Y)=X+O(X^2)$ of the functional equation $(\ref{eq:funceq_A_3})$ is a holomorphic function around $X=Y=0$.  

Note that, if the formal power series $A(X, X)$ has a positive radius of convergence, then $A(X, Y)$ is a holomorphic function around $X=Y=0$.  
This fact immediately follows from the fact that each coefficient $A_{n m}$ of $A(X, Y)$ is a non-negative real number and thus the inequality 
  \[
  A_{n, m}\leq\hskip-5mm\sum_{0\leq \nu, \mu, \nu+\mu=n+m}\hskip-5mmA_{\nu, \mu}=\widetilde{A}_{n+m}:=\text{the coefficient of }X^{n+m} \text{ in the expansion of }A(X, X)
 \]
holds. Therefore, all we have to do is to show that $A(X, X)$ has a positive radius of convergence. 
 For this purpose, consider a formal power series $B(X)=X+\sum_{n=2}^\infty B_nX^n$ defined by 
 \[
  \sum_{n=2}^\infty \varepsilon_{n-1}^{-1}B_nX^n=\frac{2KB(X)^2}{(1-RB(X))^2}. 
 \]
As it follows from Lemma \ref{lem:siegel} that $B(X)$ has a positive radius of convergence, 
it is sufficient to show that the inequality $\widetilde{A}_n\leq B_n$ holds for each $n\geq 2$. 
We will prove it by induction. 
Since $\widetilde{A}_2=\varepsilon_1K$ and $B_2=2\varepsilon_1K$ hold, 
it is clear that the inequality $\widetilde{A}_n\leq B_n$ holds for $n=2$. 
Next, let us assume that $\widetilde{A}_\nu\leq B_\nu$ holds for each $\nu<n$. 
For it follows from the equation $(\ref{eq:funceq_A_3})$ that
   \[
  \sum_{n=2}^\infty\varepsilon_{n-1}^{-1}\widetilde{A}_nX^n=\frac{K}{1-RX}\left((A(X, X)-X)X+\frac{A(X, X)^2}{1-RA(X, X)}\right), 
 \]
by using the assumption of the induction, it is reduced to showing the inequality
 \begin{eqnarray}
&&\text{the coefficient of }X^n\text{ in the expansion of }
  \frac{K}{1-RX}\left((B(X)-X)X+\frac{B(X)^2}{1-RB(X)}\right)\nonumber \\
&\leq&\text{the coefficient of }X^n\text{ in the expansion of }
\frac{2KB(X)^2}{(1-RB(X))^2}, \nonumber
 \end{eqnarray}
which is easily obtained from the fact that the coefficients of ${B}(X)-X$ and $X$ are less than or equal to that of ${B}(X)$. 
Now we have proven that $A(X, X)$ has a positive radius of convergence. 
Therefore we can construct a system $\{(W_j, \widetilde{w}_j)\}$ with $t_{jk}\widetilde{w}_k\equiv\widetilde{w}_j$ by solving the functional equation $(\ref{eq:eqfunc11})$, which completes the proof of Theorem \ref{main} in the case where the condition $(ii)$ holds. 

\subsection{Proof of Theorem \ref{main} $(iii)$}

Finally we prove Theorem \ref{main} in the case where the condition $(iii)$ holds. 
In this case, we can apply \ref{prop:Rossi} and thus we may assume that the neighborhood $V$ satisfies the following property: for each $n\geq 1$, there exists an integer $N(N_{S/X}|_V^{-n})$ such that the natural map $H^1(V, N_{S/X}|_V^{-n})\to H^1(V, N_{S/X}|_V^{-n}\otimes\mathcal{O}_{V}/I_C^m)$ is injective for each integers $n\geq 1$ and $m\geq N(N_{S/X}|_V^{-n})$, where $I_C$ is the defining ideal sheaf of $C\subset V$. 
Fix a relatively compact open domain $V_0\subset V$ which contains $C$. 

Fix a system $\{(W_j, w_j)\}$ of order $(1, 0)$ and consider again the functional equation $(\ref{eq:eqfunc11})$. 
Let $M$ and $R$ be sufficiently large positive number as in the proof of Theorem \ref{main} in the case where $(i)$ (or $(ii)$) holds. 
Fix a positive real number $K$ and consider the solution $A(X)$ of the functional equation 
 \begin{equation}\label{eq:funceq_A_3}
  A(X)-X=\frac{KR(1+MR)A(X)^2}{1-RA(X)}
 \end{equation}
 and denote by $A_{n}$ the coefficient of $X^n$ in the Taylor expansion of $A(X)$ at $X=0$. 
Though the functional equation $(\ref{eq:funceq_A_3})$ has two solutions, the solution $A$ is uniquely determined by the condition that $A(X)=X+O(X^2)$. 
Note that $A_{n}$ is a non-negative real number for each $n$. 

\begin{lemma}\label{lem:newsystem_2_torsion_3}
When $K$ is sufficiently large, 
there exists a system of functions $\{F_j^{(n)}(x_j, z_j)\}_{n\geq 2}$ for each $j$ satisfying the following conditions. 
Let 
\begin{eqnarray}
F_j^{(n)}(x_j, z_j)&=&F_j^{(n)}(x_j(x_k, z_k, w_k), z_j(x_k, z_k, w_k)) \nonumber \\
&=&F_j^{(n)}(x_j(x_k, z_k, 0), z_j(x_k, z_k, 0))
+\sum_{p=1}^\infty F_{jk, p}^{(n)}(x_k, z_k)\cdot w_k^{p}\nonumber
\end{eqnarray}
be the expansion of $F_j^{(n)}|_{W_{jk}}$ in the variable $w_k$ by regarding $F_j^{(n)}$ as a function defined on $W_{j}$ which does not depend on $w_j$. 
Denote by $h_{2jk, n}$ the coefficient of $u_j^n$ in the expansion of 
\[
\sum_{\nu=2}^{n-1}\sum_{p=1}^\infty F_{kj, p}^{(\nu)}(x_j, z_j)\cdot u_j^\nu\cdot\left(u_j+\sum_{a=2}^{n-1}F_j^{(a)}(x_j, z_j)\cdot u_j^a\right)^{p}
\]
and by $h_{3jk, n}$ the coefficient of $z_j^mu_j^n$ in the expansion of 
\[
\sum_{\nu=2}^{\infty}f_{jk}^{(\nu)}(x_j, z_j)\cdot\left(u_j+\sum_{p=2}^{n-1}F_j^{(p)}(x_j, z_k)\cdot u_j^p\right)^\nu
-\sum_{\nu=2}^{n}f_{jk}^{(\nu)}(x_j, z_j)\cdot u_j^\nu, 
\]
where $f_{jk}^{(\nu)}$ is that in the expansion $(\ref{eq:thm_tenkai})$. 
Then the coboundary $\delta\{(V_j, F_j^{(n)})\}$ is equal to 
\[
\{(V_{jk}, f_{jk}^{(n)}-t_{jk}^{-n+1}h_{2jk, n}+h_{3jk, n})\}, 
\]
 and 
$\max_j\sup_{V_0\cap V_j}\left|F_j^{(n)}\right|\leq A_{n}$ for each $n\geq 1$. 
\end{lemma}

\begin{proof}
Set $F_j^{(1)}(x_j, z_j):=1$. 
We construct the system $\{F^{(\nu)}_j\}$ inductively. 
Assume that there exists a system $\{F^{(\nu)}_j\}$ as in Lemma \ref{lem:newsystem_2_torsion_3} for each $\nu<n$. 
Let $\widetilde{w}_j$ be the solution of the new functional equation 
\[
w_j=\sum_{\nu=1}^{n-1} F_j^{(\nu)}(x_j, z_j)\cdot u_j^{\nu}. 
\]
From just the same argument as in the proof of Theorem \ref{thm:ueda_3} in \cite{U}, it follows that the new system $\{(W_j, \widetilde{w}_j)\}$ is of order $(n-1, 0)$ and 
\[
\frac{1}{n}(\widetilde{w}_j^{-(n-1)}-t_{jk}^{-(n-1)}\widetilde{w}_k^{-(n-1)})|_{V_{jk}}=f_{jk}^{(n)}-t_{jk}^{-n+1}h_{2jk, n}+h_{3jk, n}
\]
holds. 
From the assumption of Theorem \ref{main}, it holds for each integer $m$ that 
the induced cohomology class $[\{(V_{jk}, f_{jk}^{(n)}-t_{jk}^{-n+1}h_{2jk, n}+h_{3jk, n})\}]_m\in H^1(V, N_{S/X}|_V^{-(n-1)}\otimes \mathcal{O}_V/I_C^m)$ is the trivial one. 
By considering this fact especially for $m=N(N_{S/X}|_V^{-(n-1)})$, we can conclude that the cohomology class $[\{(V_{jk}, f_{jk}^{(n)}-t_{jk}^{-n+1}h_{2jk, n}+h_{3jk, n})\}]\in H^1(V, N_{S/X}|_V^{-(n-1)})$ itself is also the trivial one. 
Thus Lemma \ref {lem:newsystem_2_torsion_3} follows from just the same argument as in the proof of Theorem \ref{thm:ueda_3} in \cite{U} by using Lemma \ref{lem:KS_2} instead of \cite[Lemma 3]{U}. 
\end{proof}

Let $\{F^{(n)}_j\}$ be that in Lemma \ref{lem:newsystem_2_torsion_3}. 
Then, it can be said that $\sum_{n=2}^\infty F^{(n)}_ju_j^n$ is a holomorphic function around $V_0$. 
Let $\widetilde{w}_j$ be the solution of the functional equation $(\ref{eq:eqfunc11})$. 
Then, it follows from the same argument as in the proof of Theorem \ref{thm:ueda_3} in \cite{U} that $t_{jk}\widetilde{w}_k-\widetilde{w}_j\equiv 0$ holds for each $j$ and $k$, which shows Theorem \ref{main} when $(iii)$ holds. 

\section{Some examples and proof of Corollary \ref{main_cor}}

\subsection{$\mathbb{P}^1$-bundle examples}

\begin{example}\label{eg:1mu}
Let $S_0$ be a complex manifold and $C_0\subset S_0$ be a smooth compact K\"ahler hypersurface with topologically trivial normal bundle. 
Assume that the line bundle $\mathcal{O}_{S_0}(C_0)$ is flat around $C_0$ 
(it follows from Theorem \ref{thm:ueda_3} that it is sufficient to assume that  $N_{C_0/S_0}\in\mathcal{E}_0(C_0)\cup\mathcal{E}_1(C_0)$ and ${\rm type}\,(C_0, S_0)=\infty$). 
 Fix an extension of the trivial line bundle $\mathcal{O}_{S_0}$ by $\mathcal{O}_{S_0}(C_0)$: 
 \begin{equation}\label{exs:ext}
  0\to \mathcal{O}_{S_0}(C_0)\to E\to \mathcal{O}_{S_0}\to 0\ \text{: exact}. 
 \end{equation}
Define $X:=\mathbb{P}(E)$. 
Denote by $S$ the section of $X\to S_0$ and by $C$ the submanifold of $S$ isomorphic to $C_0$ via the natural isomorphism $S_0\to S$. 
Fix a sufficiently fine open covering $\{V_{j}\}_j$ of a sufficiently small neighborhood $V$ of $C_0$ in $S_0$. 
Then, by using a local frame $t_j$ of $\mathcal{O}_{S_0}(C_0)$ on $V_j$, a coordinate $x_j$ of $V_j\cap C_0$, and a defining function $z_j$ of $V_j\cap C_0$ in $V_j$, we can define a local coordinates system
\[
(x_j, z_j, w_j):=[(1, w_j\cdot t_j(x_j, z_j))]\in (E^*|_{V_j}\setminus (0\text{-section}))/\mathbb{C}^*\subset X
\]
of $X$. 
Note that we can choose $\{(V_j, t_j)\}$ such that $t_{jk}:=t_k/t_j\in U(1)$ holds for each $j$ and $k$ by shrinking $V$ if necessary. 
Then simple computation shows that there exists a holomorphic function $p_{jk}$ on $V_{jk}$ such that 
\begin{equation}\label{shiki1}
w_j=\frac{t_{jk}^{-1}\cdot w_k}{1+p_{jk}(x_j, z_j)\cdot w_k}
\end{equation}
holds. 
Note that the system $\{(V_{jk}, p_{jk})\}$ is a $1$-cocycle of $\mathcal{O}_{S_0}(C_0)$ and the class $[\{(V_{jk}, p_{jk})\}]\in H^1(V, \mathcal{O}_{S_0}(C_0))$ coincides with the restriction of the extension class of the exact sequence $(\ref{exs:ext})$. 
It immediately follows from $(\ref{shiki1})$ that $dw_j=t_{jk}^{-1}dw_k$ holds, and thus we obtain  $N_{S/X}^{-1}=\mathcal{O}_S(C)$ (i.e. $N_{S/X}|_C=N_{C/X}^{-1}$). 
Moreover, by expanding $(\ref{shiki1})$ in the variable $w_k$, we obtain 
\[
 t_{jk}w_j=w_k-p_{jk}(x_j, z_j)w_k^2+O(w_k^3). 
\]
Thus it can be said that, if the exact sequence $(\ref{exs:ext})$ does not split around $C_0$, 
then there exists an integer $m$ such that ${\rm type}\,(C, S, X)=(1, m)$. 
Moreover, the equation 
\[
 u_{1, \mu}(C, S, X)=\{(U_{jk}, -(p_{kj}/z_j^\mu)|_{U_{jk}})\}\in H^1(C, N_{S/X}|_C^{-1}\otimes N_{C/S}^{-\mu})
\]
holds for each integer $\mu$ less than $\min_{jk}{\rm mult}_{U_{jk}}p_{jk}$, where $U_{jk}:=V_{jk}\cap C_0$. 

On the other hand, when $(\ref{exs:ext})$ splits, we may assume $p_{jk}\equiv 0$ for each $j, k$ and thus we obtain $t_{jk}w_j=w_k$ for each $j, k$, which prove that ${\rm type}\,(C, S, X)=\infty$. 
\end{example}

\begin{example}
Let $C_1, C_2\subset \mathbb{P}^2$ be two smooth elliptic curves which intersects at $9$ points $p_1, p_2, \dots, p_9$ transversally. 
Denote by $S_0$ the blow-up of $\mathbb{P}^2$ at these points $p_1, p_2, \dots, p_9$, and by $C_0$ the strict transform of $C_1$. 
Note that $S_0$ has a structure of an elliptic surface and $C_0$ is a fiber of $S_0$, 
and thus $\mathcal{O}_{S_0}(C_0)$ is trivial around $C_0$. 
Take an extension $E$ of $\mathcal{O}_{S_0}$ by $\mathcal{O}_{S_0}(C_0)$ as follows. 
First consider the short exact sequence 
\[
 0\to \mathcal{O}_{S_0}\to \mathcal{O}_{S_0}(C_0)\to i^*\mathcal{O}_{C_0}\to 0
\]
obtained from the fact that $\mathcal{O}_{S_0}(C_0)|_{C_0}
=\mathcal{O}_{C_0}$, where $i\colon C_0\to S_0$ is the inclusion. 
Let 
\[
 H^1(S_0, \mathcal{O}_{S_0}(C_0))\to H^1(C_0, \mathcal{O}_{C_0})\to H^2(S_0, \mathcal{O}_{S_0})=0
\]
be the sequence induced from the above short exact sequence. 
From this exact sequence, it follows that there exists a non-trivial element $\xi\in H^1(S_0, \mathcal{O}_{S_0}(C_0))$. 
Let 
\[
  0\to \mathcal{O}_{S_0}(2C_0)\to E\to \mathcal{O}_{S_0}\to 0
\]
be an extension corresponds to the class 
$f_{C_0}\cdot \xi\in H^1(S_0, \mathcal{O}_{S_0}(2C_0))={\rm Ext}^1(\mathcal{O}_{S_0}, \mathcal{O}_{S_0}(2C_0))$, where $f_{C_0}\in H^0(S_0, \mathcal{O}_{S_0})$ is a canonical section. 
Define $X:=\mathbb{P}(E)$. 
Denote by $S$ the section of $X\to S_0$ and by $C\subset S$ the curve corresponds to $C_0$ via $S_0\to S$. 
Then it follows from the same argument as in Example \ref{eg:1mu} that $u_{1, 0}(C, S, X)=0$ and $u_{1, 1}(C, S, X)=\xi\not=0$, and thus it holds that ${\rm type}\,(C, S, X)=(1, 1)$. 
\end{example}

\subsection{On the blow-up of $\mathbb{P}^3$ at $8$ points}

Corollary \ref{main_cor} can be regarded as an analogue of an application of Ueda's theory on the blow-up of $\mathbb{P}^2$ at $9$ points. 
First we review this result on the blow-up of $\mathbb{P}^2$ at $9$ points. 
Let $p_1, p_2, \dots p_9$ be general $9$ points of $\mathbb{P}^2$. Then there exists an unique elliptic curve $C_0\subset \mathbb{P}^2$ passing through all of these points. Assume that this curve $C_0$ is smooth. 
Denote by $N$ the line bundle $\mathcal{O}_{\mathbb{P}^2}(3)|_{C_0}\otimes\mathcal{O}(-p_1-p_2-\dots-p_9)$. Clearly it can be said that $N$ is an element of ${\rm Pic}^0(C_0)$ and thus $N$ is flat, since $C_0$ is K\"ahler. 
When $N$ is a torsion element, the canonical bundle $K_X^{-1}$ of the blow-up $X$ of $\mathbb{P}^2$ at $\{p_1, p_2\dots, p_9\}$ is semi-ample. 
Especially, in this case, $K_X^{-1}$ is semi-positive: i.e. $K_X^{-1}$ admits a smooth Hermitian metric with semi-positive curvature. 
On the other hand, when $N$ is a non-torsion element, $K_X^{-1}$ is not semi-ample. 
In this case, Brunella showed that $K_X^{-1}$ is semi-positive if and only if $C$ has a pseudoflat neighborhood in $X$, 
where $C$ is the strict transform of $C_0$ \cite{B}. 
As this condition holds if the line bunlde $\mathcal{O}_X(C)$ is flat around $C$, 
it follows from this Brunella's theorem and Theorem \ref{thm:ueda_3} that $K_X^{-1}$ is semi-positive if  $N\in \mathcal{E}_1(C_0)$. 

\begin{table}[tbh]
\begin{center}
  \begin{tabular}{|c||c|c|}\hline
     & $N$ : torsion & $N$ : non-torsion  \\ \hline
  the base locus $\mathbb{B}(K_X^{-1})$ & $\emptyset$ or $C$  & $C$   \\ \hline
  semi-ampleness &  semi-ample & not semi-ample   \\ \hline
  Iitaka dimension & 1 & 0   \\ \hline
  \end{tabular}
  \caption{Properties of the anti-canonical bundle of the blow-up $X$ of $\mathbb{P}^2$ at $9$ points}
  \label{table:p2_9}
  \end{center}
\end{table}

Now let us start considering the blow-up of $\mathbb{P}^3$ at $8$ points. 
Fix general $8$ points $p_1, p_2, \dots, p_8$ of $\mathbb{P}^3$. 
Then there exits an $1$-dimensional family of quadric surfaces of $\mathbb{P}^3$ passing through these points $p_1, p_2, \dots, p_8$. 
Fix such quadric surfaces $Q_0$ and $Q_\infty$ of $\mathbb{P}^3$. 
Assume that $Q_0$ and $Q_\infty$ intersect each other transversally along $C_0:=Q_0\cap Q_\infty$. 
Note that $C_0$ is a smooth elliptic curve and $\mathcal{O}_{Q_0}(C_0)=K_{Q_0}^{-1}$ holds in this case. 
Denote by $X$ the blow-up of $\mathbb{P}^3$ at $\{p_1, p_2\dots, p_8\}$, by $C$ the strict transform of $C_0$, and by $S_0$ (resp. $S_\infty$) the strict transform of $Q_0$ (resp. $Q_\infty$). 
Note that $K_{X}^{-1}=\mathcal{O}_X(2S_0)=\mathcal{O}_X(2S_\infty)$, 
$N_{S_0/X}=\mathcal{O}_{S_0}(C)$, and that $N_{C/S_0}$ is isomorphic to $N:=\mathcal{O}_{\mathbb{P}^3}(2)|_{C_0}\otimes\mathcal{O}(-p_1-p_2-\dots -p_8)$ via the natural isomorphism $C\to C_0$. 
When $N\in{\rm Pic}^0(C)$ is a torsion element, $K_X^{-1}$ is semi-ample,  
and thus it is semi-positive. 
On the other hand, when $N$ is a non-torsion element, the base locus $\mathbb{B}(K_X^{-m})$ is equal to $C$ for each $m\geq 1$ and thus $K_X^{-1}$ is not semi-ample. 

\begin{table}[tbh]
\begin{center}
  \begin{tabular}{|c||c|c|}\hline
     & $N$ : torsion & $N$ : non-torsion  \\ \hline
  the base locus $\mathbb{B}(K_X^{-1})$ & $\emptyset$ or $C$  & $C$   \\ \hline
  semi-ampleness &  semi-ample & not semi-ample   \\ \hline
  Iitaka dimension & 2 & 1   \\ \hline
  \end{tabular}
  \caption{Properties of the anti-canonical bundle of the blow-up $X$ of $\mathbb{P}^3$ at $8$ points}
  \label{table:p3_8}
  \end{center}
\end{table}

\begin{proof}[Proof of Corollary \ref{main_cor}]
We apply Theorem \ref{main} on the triple $(C, S_0, X)$. 
For this purpose, we show that $N_{S_0/X}=\mathcal{O}_X(S_0)|_{S_0}=\mathcal{O}_X(S_\infty)|_{S_0}=\mathcal{O}_{S_0}(C)$ is flat on a neighborhood of $C$ in $S_0$ by applying Ueda's theory on the pair $(C, S_0)$. 
From the assumption, $N_{C/S_0}$ is an element of $\mathcal{E}_1(C)$. 
Note that $H^1(C, N_{C/S_0}^{-n})=0$ holds for $n\geq 1$, since $C$ is an elliptic curve and $N_{C/S_0}$ is non-torsion. 
Thus $u_n(C, S_0)=0$ holds for all $n\geq 1$ and then it follows from Theorem \ref{thm:ueda_3} that 
$N_{S_0/X}=\mathcal{O}_{S_0}(C)$ is flat on a neighborhood of $C$ in $S_0$. 

As the triple $(C, S_0, X)$ enjoys the condition $(ii)$ in Theorem \ref{main} and it follows from just the same argument on the cohomology classes as above that $u_{n, m}(C, S_0, X)=0$, we can apply Theorem \ref{main} to conclude that $\mathcal{O}_X(S_0)$ is flat on a neighborhood $W$ of $C$ in $X$. 
As a conclusion, it can be said there exists a flat metric $h_1$ on the line bundle $K_X^{-1}|_W$.  

Let $f_0\in H^0(X, S_0)$ and $f_\infty\in H^0(X, S_\infty)$ be canonical sections. 
Denote by $h_2$ the singular Hermitian metric defined by two sections $f_0^2, f_\infty^2\in H^0(X, K_X^{-1})$: $h_2=(|f_0|^2+|f_\infty|^2)^{-1}$. 
Clearly $h_2$ has a semi-positive curvature current on the whole $X$ and a singularity along $C$. 
We can construct a smooth Hermitian metric on $K_X^{-1}$ with semi-positive curvature by taking the ``regularized minimum'' of two metrics $M\cdot h_1$ and $h_2$, where $M$ is a sufficiently large real number (see the proof of \cite[3.4]{K2} for the precise meaning of the ``regularized minimum'' of the singular Hermitian metrics). 
\end{proof}

%%%%%%%%%%%% References %%%%%%%%%%%%%


\begin{thebibliography}{99}
 \bibitem[B]{B} \textsc{M. Brunella}, On K\"ahler surfaces with semipositive Ricci curvature, Riv. Mat. Univ. Parma, {\bf 1} (2010), 441--450. 
 \bibitem[CM]{CM} \textsc{C. Camacho, H, Movasati}, Neighborhoods of Analytic Varieties, 
Monograf\'\i as del Instituto de Matem\'atica y Ciencias Afines, 35. Instituto de Matem\'atica y Cienc\'\i as Afines, IMCA, Lima; Pontificia Universidad Cat\'olica del Per\'u, Lima, 2003. 
% \bibitem[D]{D} \textsc{J.-P. Demailly}, Analytic methods in algebraic geometry, Surveys of Modern Mathematics, {\bf 1}. International Press, Somerville, MA, 2012. 
% \bibitem[DPS]{DPS94} \textsc{J-P. Demailly, T. Peternell, M. Schneider}, Compact complex manifolds with numerically effective tangent bundles, J. Algebraic Geom., {\bf 3} (1994), 295-345. 
 \bibitem[G]{G} \textsc{H. Grauert}, \"{U}ber Modifikationen und exzeptionelle analytische Mengen, Math. Ann., {\bf 146} (1962), 331--368. 
% \bibitem[HR]{HR} \textsc{H. Hironaka, H. Rossi}, On the equivalence of imbeddings of exceptional
%complex spaces, Math. Ann., {\bf 156} (1964), 313--333.
 \bibitem[KS]{KS} \textsc{K. Kodaira and D. C. Spencer}, A theorem of completeness of characteristic systems of
complete continuous systems, Amer. J. Math., {\bf 81} (1959), 477-500. 
\bibitem[K1]{K2} \textsc{T. Koike}, On minimal singular metrics of certain class of line bundles whose section ring is not finitely generated,
to appear in Ann. Inst. Fourier (Grenoble), 
arXiv:1312.6402. 
 \bibitem[K2]{K3} \textsc{T. Koike}, On the minimality of canonically attached singular Hermitian metrics on certain nef line bundles, 
to appear in Kyoto J. Math., 
arXiv:1405.4698. 
 \bibitem[L]{L} \textsc{H. B. Laufer}, Normal two-dimensional singularities, Ann. Math. Stud., {\it 71}. Princeton University Press, Princeton, N.J.; University of Tokyo Press,
Tokyo, 1971. 
 \bibitem[LO]{LO} \textsc{J. Lesieutre, J. C. Ottem}, Curves disjoint from a nef divisor, arXiv:1410.4467. 
 \bibitem[N]{N} \textsc{A. Neeman}, Ueda theory: theorems and problems, Mem. Amer. Math. Soc. {\bf 81} (1989), no. 415. 
 \bibitem[R]{R} \textsc{H. Rossi}, Strongly pseudoconvex manifolds, Lectures in Modern Analysis and Applications I, Lecture Notes in Mathematics {\bf 103}, (1969), 10--29. 
\bibitem[S]{S} \text{C. L. Siegel}, Iterations of analytic functions, Ann. of Math., {\bf 43} (1942), 607--612. 
\bibitem[T]{T} \text{B. Totaro}, Moving codimension-one subvarieties over finite fields, Amer. J. Math. {\bf 131} (2009), 1815--1833.
 \bibitem[U]{U} \textsc{T. Ueda},  On the neighborhood of a compact complex curve with topologically trivial
normal bundle, Math. Kyoto Univ., {\bf 22} (1983), 583--607. 
\end{thebibliography}
\end{document}